\documentclass[12pt,a4paper]{amsart}

\usepackage[utf8]{inputenc}
\usepackage[T1]{fontenc}
\usepackage[english]{babel}
\usepackage[margin=2cm]{geometry}
\usepackage{enumitem}
\usepackage{caption}
\usepackage{subcaption}

\usepackage[colorlinks]{hyperref}
\usepackage{tikz}

\usepackage{amsmath,amssymb,amsfonts,amsthm}
\usepackage{mathtools,mathrsfs}
\usepackage{stmaryrd}
\SetSymbolFont{stmry}{bold}{U}{stmry}{m}{n}
\usepackage[capitalize,nameinlink]{cleveref}
\allowdisplaybreaks{}

\definecolor{coGB}{HTML}{1F77B4}
\definecolor{coWG}{HTML}{2CA02C}
\definecolor{coAR}{HTML}{D62728}

\hypersetup{
    colorlinks = True,
    linkcolor  = coAR,
    citecolor  = coWG,
    urlcolor   = coGB,
}

\newcommand{\Cbb}{\mathbb{C}}
\newcommand{\Hbb}{\mathbb{H}}

\newcommand{\Lbb}{\mathbb{L}}
\newcommand{\Nbb}{\mathbb{N}}
\newcommand{\Rbb}{\mathbb{R}}
\newcommand{\Tbb}{\mathbb{T}}
\newcommand{\Zbb}{\mathbb{Z}}
\newcommand{\Irm}{\mathrm{I}}
\newcommand{\Lrm}{\mathrm{L}}

\newcommand{\Cscr}{\mathscr{C}}
\newcommand{\Hscr}{\mathscr{H}}
\newcommand{\Kscr}{\mathscr{K}}
\newcommand{\Lscr}{\mathscr{L}}

\newcommand{\Sscr}{\mathscr{S}}

\renewcommand{\epsilon}{\varepsilon}
\newcommand{\eps}{\varepsilon}

\newcommand{\ex}{\mathsf{e}}
\newcommand{\ic}{\mathsf{i}\mkern1mu}

\newcommand{\vect}[1]{{\boldsymbol{#1}}}

\newcommand{\Ho}{\mathop{}\!\mathsf{H}^{(1)}}

\DeclareMathOperator{\bJ}{\mathsf{J}}

\DeclareMathOperator{\ce}{\mathsf{ce}}
\DeclareMathOperator{\se}{\mathsf{se}}
\newcommand{\Mce}[1]{\mathop{}\!\mathsf{Mc}^{(#1)}}
\newcommand{\Mse}[1]{\mathop{}\!\mathsf{Ms}^{(#1)}}
\DeclareMathOperator{\argth}{artanh}

\newcommand{\plr}[1]{\left(#1\right)}
\newcommand{\abs}[1]{\left\lvert#1\right\rvert}
\newcommand{\norm}[1]{\left\lVert#1\right\rVert}

\newcommand{\setst}[2]{\left\lbrace#1\ \middle\vert\ #2\right\rbrace}

\newcommand{\di}[1]{\mathop{}\!\mathrm{d}#1}
\DeclareMathOperator{\OO}{\mathcal{O}}
\DeclareMathOperator{\oo}{\mathcal{\scriptstyle{} O}}
\DeclareMathOperator{\diag}{diag}

\newcommand{\inc}{\mathsf{in}}
\newcommand{\sca}{\mathsf{sc}}
\newcommand{\eve}{\mathsf{ev}}
\newcommand{\odd}{\mathsf{od}}

\newcommand{\ana}{\mathrm{ana}}
\newcommand{\asy}{\mathrm{asy}}

\theoremstyle{definition}
\newtheorem{definition}{Definition}

\theoremstyle{plain}
\newtheorem{lemma}[definition]{Lemma}

\theoremstyle{remark}
\newtheorem{remark}[definition]{Remark}

\begin{document}

\title[Quadrature by Parity Asymptotic eXpansions (QPAX)]{Quadrature by Parity Asymptotic eXpansions (QPAX) for scattering by high aspect ratio particles}

\author[C. Carvalho]{Camille Carvalho}
\address{Department of Applied Mathematics, University of California, Merced, 5200 North Lake Road, Merced, CA 95343, USA.}
\email{ccarvalho3@ucmerced.edu}

\author[A. D. Kim]{Arnold D.~Kim}
\address{Department of Applied Mathematics, University of California, Merced, 5200 North Lake Road, Merced, CA 95343, USA.}
\email{adkim@ucmerced.edu}

\author[L. Lewis]{Lori Lewis}
\address{Mathematics Department, Santa Rosa Junior College, 1501 Mendocino Avenue, Santa Rosa, CA 95401, USA.}
\email{llewis@santarosa.edu}

\author[Z. Moitier]{Zo{\"\i}s Moitier}
\address{Karlsruhe Institute of Technology, Institute for Analysis, Englerstra{\ss}e 2, D-76131 Karlsruhe, Germany.}
\email{zois.moitier@kit.edu}

\date{\today}
\thanks{This research was funded by the NSF (DMS-1819052 and DMS-1840265) and by the Deutsche Forschungsgemeinschaft (DFG, German Research Foundation) --- Project-ID 258734477 --- SFB 1173.}



\begin{abstract}
    We study scattering by a high aspect ratio particle using boundary integral equation methods.
    This problem has important applications in nanophotonics problems, including sensing and plasmonic imaging.
    To illustrate the effect of parity and the need for adapted methods in presence of high aspect ratio particles, we consider the scattering in two dimensions by a sound-hard, high aspect ratio ellipse.
    This fundamental problem highlights the main challenge and provide valuable insights to tackle plasmonic problems and general high aspect ratio particles.
    For this problem, we find that the boundary integral operator is nearly singular due to the collapsing geometry from an ellipse to a line segment.
    We show that this nearly singular behavior leads to qualitatively different asymptotic behaviors for solutions with different parities.
    Without explicitly taking this nearly singular behavior and this parity into account, computed solutions incur a large error.
    To address these challenges, we introduce a new method called Quadrature by Parity Asymptotic eXpansions (QPAX) that effectively and efficiently addresses these issues.
    We first develop QPAX to solve the Dirichlet problem for Laplace's equation in a high aspect ratio ellipse.  Then, we extend QPAX for scattering by a sound-hard, high aspect ratio ellipse.
    We demonstrate the effectiveness of QPAX through several numerical examples.
\end{abstract}

\maketitle

\setcounter{tocdepth}{1}
\tableofcontents

\section{Introduction}

Metal nanoparticles are being used extensively for chemical and biological sensing applications because they exhibit strong electromagnetic field responses and they are biologically and chemically inert (\emph{e.g.}\ see~\cite{anker2008biosensing} for a review).
For these applications, the shape of individual metal nanoparticles can drastically affect the sensitivity of sensors.
Consequently, there has been much interest in understanding how nanoparticle shape affects scattering by electromagnetic fields.
In particular, there has been interest in studying so-called high aspect ratio nanoparticles~\cite{bauer2004biological, paivanranta2011, avolio2019elongated, ruiz2019slender,deng2021mathematical, ruiz2021plasmonic}.
Examples of these include nano-rods, nano-wires, nano-tubes, etc.
High aspect ratio nanoparticles have anisotropic and tunable chemical, electrical, magnetic, and optical properties that make them attractive for designing sensors.
Moreover, the diffusion of high aspect ratio nanoparticles provides enhanced functionality such as delivering DNA in plant cells~\cite{demirer2019}.

It is then fundamental to understand the scattering properties of high
aspect ratio nanoparticles to design next-generation sensors.
However, these scattering problems are challenging due to the inherently high anisotropy affecting both the near- and far-field behaviors of the scattered field.
Indeed, the narrow axis of a high aspect ratio nanoparticle may be much smaller than the wavelength of the incident field, while the long axis may be comparable to or larger than it.

Motivated by these sensing applications, we seek to study the fundamental issues inherent in scattering by high aspect ratio particles.
To that end, we study here a simple model of two dimensional scalar
wave scattering by a high aspect ratio ellipse.
This problem contains several key features that make scattering
by high aspect ratio particles challenging, but is simple enough to
allow for a rigorous analysis. We will also show that the analysis
for this simple problem extends to more general problems.
We study this problem using boundary integral equation methods~\cite{Guenther1996, McLean2000, Kress2014}.
We are interested in using boundary integral equations to study this problem because they provide high accuracy~\cite{Lamp1985, Kress1990, Lee1995, Diethelm1996, Cheng1997, Helsing2008, Yang2012}, they provide valuable physical insight over all scales of the problem, and they generalize easily to more complex problems including multiply connected domains, \emph{e.g.}\ ensembles of high aspect ratio particles.

The two-dimensional problem we study here is advantageous because we
can compute its analytical solution in terms of special functions (see
\cref{sec:Appendix-A}).  We use this analytical solution to validate
our methods.  Moreover, Geer has studied scattering by slender bodies
using asymptotic analysis and obtained a uniformly valid asymptotic
solution~\cite{Geer1978}.  That asymptotic analysis can be applied to
this problem and adds further insight.  Essentially, Geer's
asymptotic solution for this scattering problem is given by a
continuous distribution of point sources over a line segment along the
semi-major axis of the ellipse.  This asymptotic analysis allows us to
anticipate the challenges that the boundary integral equation
formulation will have with this problem.

In this paper we show that there is an underlying nearly singular
behavior within the scattering boundary integral operator, in the
limit where the ellipse coalesces to a line segment.  Nearly singular
behaviors arise when the kernel of the integral operator is sharply
peaked (but not singular), leading to a large error in computations.
Typically, we see nearly singular behaviors in the so-called close
evaluation problem, corresponding to evaluating layer potentials near
the boundary~\cite{Guenther1996, Helsing2008, Sladek2000, Beale2016, Carvalho2018, Carvalho2020}.  In other words, the close evaluation
problem happens \emph{after} one has accurately
solved the boundary integral equation.  For this present problem,
the high aspect ratio ellipse leads to a nearly singular integral
operator in the boundary integral equation, itself, thereby adversely affecting the accuracy
of the computed solution.  Nearly singular boundary integral
equations also arise when boundaries in a multiply-connected domain
are close to one another.  This nearly singular behavior was seen when
particles suspended in Stokes flow were situated closely to one
another~\cite{ludvig2016,tornberg2020accurate,bagge2020highly}.  To our knowledge, there is not a
systematic treatment of nearly singular boundary integral equations.
Thus, this study provides valuable insight into those problems as
well.

By exploring the causes of the emerging nearly singular boundary
integral operator, we find that it is due to a
factor that is the kernel of the double-layer potential for
Laplace's equation.  Thus, we study the related interior Dirichlet
problem for Laplace's equation and find that there exists qualitative
differences in the asymptotic behaviors of solutions with different
parity (with respect to the major axis of the ellipse).  By addressing
this difference in parity explicitly, we develop a new method which we
call Quadrature by Parity Asymptotic eXpansions (QPAX).  We show that
this method effectively solves the interior Dirichlet problem for
Laplace's equation.  Then we extend QPAX for the scattering problem
and show that it effectively addresses the nearly singular behavior of
the associated boundary integral operator.

The remainder of this paper is organized as follows.
\Cref{sec:scattering-problem} presents the problem of scattering by a
high aspect ratio ellipse that we study here and its boundary integral
equation formulation.
This formulation reveals the cause of the nearly singular behavior.
To isolate the nearly singular behavior of the boundary integral operator identified in \Cref{sec:scattering-problem}, we study the related Dirichlet problem for Laplace's equation in a high aspect ratio ellipse in \Cref{sec:laplace}.
Here we identify that the nearly singular behavior leads to qualitatively different asymptotic behaviors based on parity.
We then introduce Quadrature by Parity Asymptotic eXpansions (QPAX) to effectively and efficiently solve that problem.
Finally, we extend QPAX for scattering by a sound-hard, high aspect
ratio ellipse in \Cref{sec:QPAX}.
In \Cref{sec:extensions} we discuss application of this method to other scattering problems including more general high aspect ratio particles shapes.
\Cref{sec:conclusion} gives our conclusions.
\Cref{sec:Appendix-A} provides details about the exact solutions used
for validation in our numerical results, and \cref{sec:Appendix-B}
gives proofs of the considered asymptotic expansions in the paper.

\section{Scattering by a high aspect ratio ellipse}%
\label{sec:scattering-problem}

We introduce a simple, two-dimensional model for studying scattering
of scalar waves by a high aspect ratio particle.
Let \( D \subset \Rbb^2 \) be a bounded simply connected set denoting the support of a particle with smooth boundary \( \partial D \) and let \( \overline{D} = D \cup \partial D \).
We study the sound-hard scattering problem, which can be written as
\begin{equation}
    \begin{aligned}
         & \text{Find \( u =  u^\inc + u^\sca \in \Cscr^2(E \coloneqq \Rbb^2 \setminus
            \overline{D}) \cup \Cscr^1(\Rbb^2 \setminus {D}) \) such that: }
        \\
         & \begin{dcases}
            \Delta u + k^2 u = 0 & \text{in } E,          \\
            \partial_{n} u = 0   & \text{on } \partial D, \\
            \lim \limits_{r \to \infty} \int_{|x| = r} \abs{\partial_n
                u^\sca - \ic k u^\sca}^2 \di{\sigma} = 0,
        \end{dcases}
    \end{aligned}\label{eq:sound-hard}
\end{equation}
where \( u \) denotes the total field, \( k \) is the wavenumber, \( u^\sca \) is the scattered field, \( u^\inc \) is the incident field, and \( \partial_n \) denotes the normal derivative.
The last equation in \cref{eq:sound-hard} is the Sommerfeld radiation condition. To study scattering by a high aspect ratio particle, we consider the ellipse defined according to
\begin{equation}
    y(t) = (\eps \cos(t), \sin (t)),
    \quad t \in \Tbb \coloneqq \Rbb / 2\pi \Zbb,
    \label{eq:ellipse}
\end{equation}
with \( 0 < \eps \ll 1 \).
The aspect ratio for this ellipse is \( \eps^{-1} \).
Hence, we study the asymptotic limit, \( \eps \to 0^+ \) for high aspect ratio particles.
Given \( x \in E \), we denote \( x^b \in \partial D \) the closest point on the boundary, and we write \( x^b = y(s) \) for \( s \in \Tbb \) (see \cref{fig:1}).

%
%
\begin{figure}[hbt]
    \centering
    \includegraphics[width=0.35\linewidth]{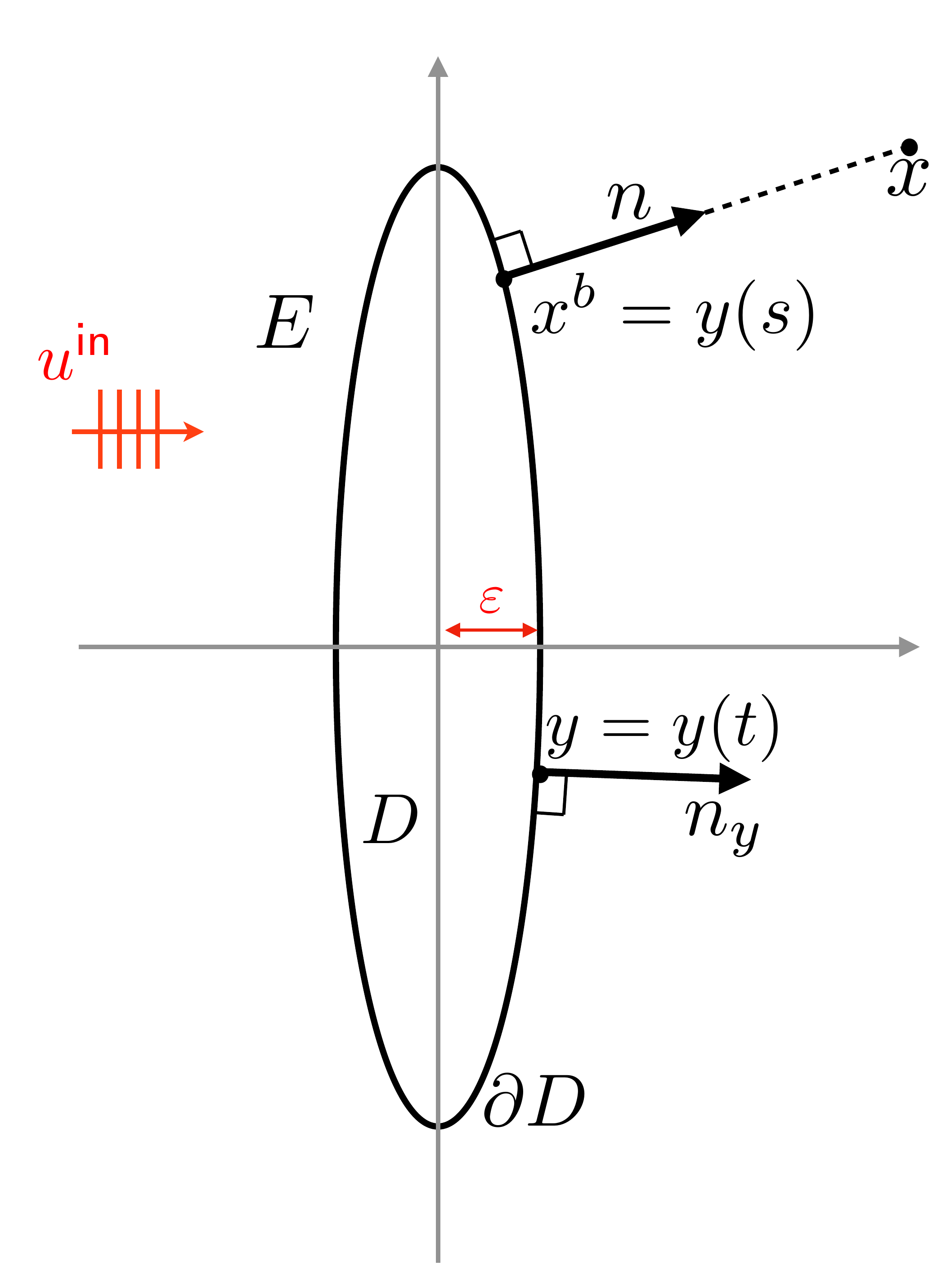}
    \caption{Sketch and notations of the scattering by a high aspect ratio ellipse problem.}%
    \label{fig:1}
\end{figure}

We can write the analytical solution of \cref{eq:sound-hard} in terms of angular and radial Mathieu functions (see \cref{eq_u_ana} in \cref{sec:Appendix-A} for details), and use that analytical solution to study the behavior of fields scattered by the high aspect ratio ellipse.
In particular, we evaluate this analytical solution using the following incident field,
\begin{equation}\label{eq:u_inc_mat}
    u^\inc(\xi, \eta)
    = \sum_{m = 0}^{15} \alpha_m^\inc \Mce{1}_m(\xi, q) \ce_m(\eta, q)
    + \sum_{m  = 1}^{15} \beta_m^\inc \Mse{1}_m(\xi, q) \se_m(\eta, q),
\end{equation}
with \( (\ce_m, \se_m) \) denoting the angular Mathieu functions of
order \( m \) and \( (\Mce{1}_{m}, \Mse{1}_{m}) \) denoting the radial
Mathieu function of the first-kind and order \( m \), both defined using
elliptical coordinates \( (\xi,\eta) \) (see
\cref{sec:Appendix-A} for more details).  The coefficients are
\( \alpha_m^\inc = 2\, \ic^m\, \ce_m(\frac{\pi}{2}, q) \) and
\( \beta_m^\inc = 2\, \ic^m\, \se_m(\frac{\pi}{2}, q) \), and
\cref{eq:u_inc_mat} approximates a plane wave propagating in the
\( +\hat{x} \) direction (see~\cite[Sec.~28.28(i)]{Nist}).  In
\cref{fig:2}, we plot the real part of \( u^{\inc} \) given by
\cref{eq:u_inc_mat} for \( k = 2 \).  We observe in \cref{fig:2} that
this incident field closely approximates a plane wave propagating in
the \( +\hat{x} \) direction within the window
\( [-5, 5] \times [-5,5] \), which provides bounds for our
computational domain.

Using \cref{eq:u_inc_mat} as the incident field, we evaluate the analytical solution for the sound-hard scattering problem, and we plot its amplitude in \cref{fig:3} for \( k = 2 \) and \( \eps = 0.01 \).
Note that with this choice of parameters, the semi-major axis of the ellipse is on the order of the wavelength, but the semi-minor axis is much smaller than the wavelength. The black bar in \cref{fig:3} represents the high aspect ratio ellipse that cannot otherwise be seen.

%
%
\begin{figure}[hbt]
    \centering
    \begin{subfigure}{0.49\textwidth}
        \centering
        \includegraphics{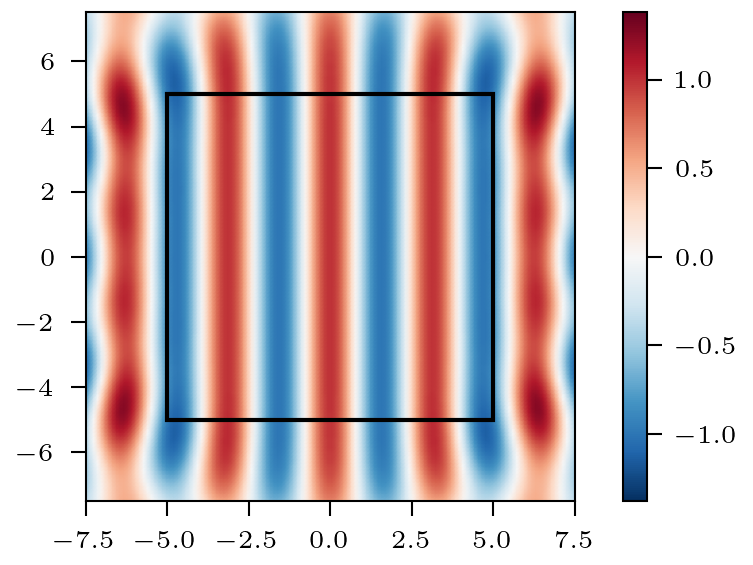}
        \caption{Real part of \( u^\inc \) defined in \cref{eq:u_inc_mat}.}%
        \label{fig:2}
    \end{subfigure}
    \begin{subfigure}{0.49\textwidth}
        \centering
        \includegraphics{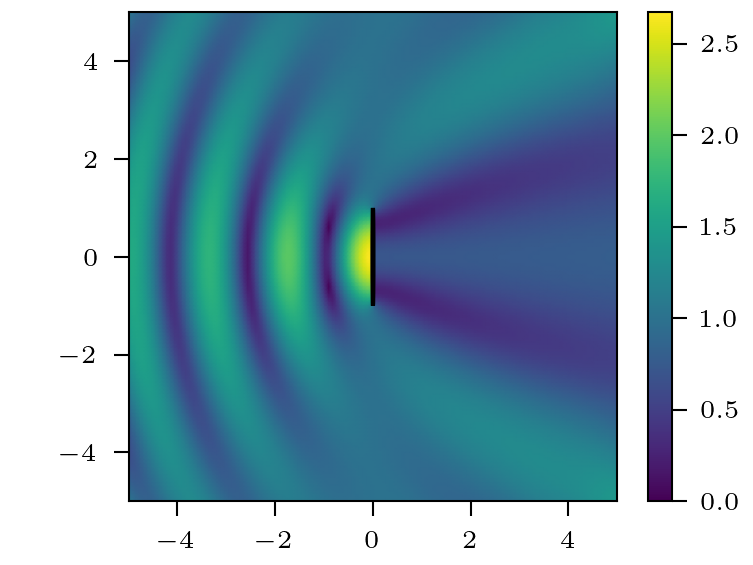}
        \caption{Amplitude of \(u\) solution of \cref{eq:sound-hard}.}%
        \label{fig:3}
    \end{subfigure}
    \caption{{Plots of the real part of the incident field \( u^\inc \) and the amplitude of the total field \( u \).}}
\end{figure}

We find that scattering by the high aspect ratio ellipse is highly anisotropic.
We observe a strong diffraction behavior around the high curvature regions, the solution demonstrates a strong amplitude near the illuminated face of the ellipse and a shadow regions on the other side of the ellipse.
It is clear from this result that the scattered field on or near the boundary plays an important role in this scattering problem, with impact on the far-field.

The boundary value problem above provides a simple setting for identifying and addressing inherent challenges arising from scattering by a high aspect ratio particle. In what follows, we incorporate asymptotic and numerical analysis to develop a method that accurately solves this problem. In \cref{sec:extensions}, we show how this method can be extended to study scattering by a sound-soft ellipse, a penetrable ellipse (in particular for plasmonics), and a particle with a more general shape.

\subsection{Boundary integral equation formulation}

Solution of the Helmholtz's equation in the exterior domain \( E \) is given by the representation formula~\cite{Kress2014}
\begin{equation}\label{eq:u}
    u(x) = u^\inc(x) + \int_{\partial D} \partial_{n_y} G(x,y) u(y) \di{\sigma_{y}}
    - \int_{\partial D} G(x,y) \partial_{n_y}u(y) \di{\sigma_{y}},
    \quad x \in E,
\end{equation}
with \( {n_y} \) denoting the outward normal at \( y \) (see \cref{fig:1}).
The fundamental solution is given by
\[
    G(x,y) = \frac{\ic}{4} \Ho_0(k |x-y|)
\]
with \( \Ho_0 \) denoting the Hankel function of the first kind of order zero.
Since \( \partial_{n} u = 0 \) on \( \partial D \),  we find that \cref{eq:u} reduces to
\begin{equation}\label{eq:u2}
    u(x) = u^\inc(x)
    + \int_{\partial D} \partial_{n_y} G(x,y) u(y) \di{\sigma_{y}},
    \qquad x \in E.
\end{equation}
{Using classic properties of the double-layer potential~\cite{Kress2014}, the unknown field \( u \) on the boundary} solves the boundary integral equation
\begin{equation}\label{eq:ubie}
    \frac{1}{2} u(x^b) - \int_{\partial D} \partial_{n_y}
    G(x^b,y) u(y) \di{\sigma_{y}} = u^\inc(x^b),
    \qquad x^b \in \partial D.
\end{equation}

On the high aspect ratio ellipse defined in \cref{eq:ellipse}, we can write the integral in \cref{eq:ubie} as
\[
    \int_{\partial D} \partial_{n_y} G(x^b, y)\ u(y) \di{\sigma_{y}}
    = \int_{\Tbb} \partial_{n_y} G(y(s), y(t))\ u(y(t))\ |y'(t)| \di{t},
\]
we then rewrite the kernel
\begin{equation}\label{eq:kernel-soundhard}
    \partial_{n_y} G(y(s), y(t)) |y'(t)|
    = \frac{\ic k \pi}{2} r(s,t;\eps) \Ho_1(k\, r(s,t;\eps))\, K^L(s,t;\eps),
\end{equation}
with
\[
    r(s,t;\eps) = 2 \abs{ \sin\plr{ \frac{s-t}{2} } }
    \sqrt{ \cos\plr{ \frac{s+t}{2} }^{2} + \eps^{2} \sin\plr{ \frac{s + t}{2} }^{2} },
\]
and
\begin{equation}\label{eq:KL}
    K^L(s,t;\eps) = \frac{1}{2\pi} \frac{-\eps}{1 + \eps^{2} + ( 1 - \eps^{2} ) \cos(s + t)}.
\end{equation}
The kernel \( K^{L} \) given in \cref{eq:KL} corresponds to the kernel for the double-layer potential solution of Laplace's equation for this narrow ellipse.
In light of \cref{eq:kernel-soundhard}, we make the following remarks.
\begin{itemize}
    \item Due to the factor of \( \Ho_{1}(k\, r(s,t;\eps)) \) appearing in the kernel given in \cref{eq:kernel-soundhard}, the derivative of the kernel has a logarithmic singularity on \( s = t \).
          This singularity is well understood (see~\cite{Kress2014}).

    \item For \( t_s \) such that \( s + t_s \equiv \pi \, [2\pi] \) (with the notation  \( k \, [X] \) denoting \( k \) modulo \( X \)), \( K^L(s,t ; \eps) \) given in \cref{eq:KL} behaves according to
          \[
              K^L(s, t_s; \eps) = -\frac{1}{4 \pi \eps}
              \quad \underset{\eps \to 0^+}{\longrightarrow} -\infty.
          \]
          Since
          \[
              \abs{ \frac{\ic k \pi}{2}\, r(s,t_{s};\eps)\, \Ho_{1}(k\, r(s,t_{s};\eps)) }
              \underset{\eps \to 0^+}{\longrightarrow} 1,
          \]
          it follows that the asymptotic behavior for \cref{eq:kernel-soundhard} on \( t_{s} \) is given by this asymptotic behavior for \( K^{L} \).

    \item The kernel \( K^L \) is sharply peaked at the \emph{mirror points} \( s+t \equiv \pi \, [2\pi] \), and this peak is enhanced as the ellipse collapses, in other words when the mirror points become closer.
          This sharp peak leads to a nearly singular integral operator in the boundary integral equation.

    \item The cases \( s = \left\lbrace -\frac{\pi}{2}, \frac{\pi}{2} \right\rbrace \) are degenerate since the derivative of the kernel \cref{eq:kernel-soundhard} admits a singularity for \( s = t \), and we have the nearly singular behavior (the mirror points coincide in this case).
\end{itemize}

\bigskip

The integral operator in the boundary integral equation \cref{eq:ubie} is weakly singular on \( s = t \) and nearly singular on \( s + t \equiv \pi[2\pi] \).
The weak singularity on \( s = t \) can be addressed using the product quadrature rule due to Kress~\cite{Kress1991}.
The nearly singular behavior, however, is problematic.
In the results that follow, we show that this nearly singular behavior leads to large error unless it is explicitly addressed.

\section{Dirichlet problem for Laplace's equation in a high aspect ratio ellipse}%
\label{sec:laplace}

In the previous section, we identified that the factor of \( K^L \) given in \cref{eq:KL} causes nearly singular behaviors in the boundary integral equation.
To isolate this issue, we consider here the following interior Dirichlet problem for Laplace's equation,
\begin{equation}\label{eq:laplace}
    \begin{aligned}
         & \text{Find \( u \in \Cscr^2(D) \cup \Cscr^1(\overline{D}) \) such that: }
        \\
         & \phantom{\{} \Delta u  = 0 \text{ in } D , \quad  u =  - f \text{ on } \partial D,
    \end{aligned}
\end{equation}
where the prescribed boundary data \( -f \) is smooth.
The solution of \cref{eq:laplace} can be represented as the double-layer
potential~\cite{Guenther1996,Kress2014}:
\[
    u(x) = \frac{1}{2 \pi} \int_{\partial D} \frac{n_y \cdot
        (x-y)}{|x-y|^2} \, \mu(y) \di{\sigma_{y}}, \quad x \in D,
\]
with \( \mu  \in \mathscr{C}^2 (\partial D) \)\footnote{In general, continuous function is sufficient. As we will see later we will need \( \mathscr{C}^2 \) regularity.} denoting the solution of the following boundary integral equation,
\begin{equation}\label{eq:bieLap}
    \frac{1}{2}\mu(x^b) -  \frac{1}{2 \pi} \int_{\partial D } \frac{n_y \cdot
        (x^b-y)}{|x^b-y|^2} \, \mu(y) \di{\sigma_{y}} = f(x^b),
    \quad x^b \in \partial D.
\end{equation}
Note that \cref{eq:bieLap} is similar to \cref{eq:ubie}. On the high aspect ratio ellipse \cref{eq:ellipse}, this boundary integral equation becomes
\begin{equation}\label{eq:ubie_param_lap}
    \displaystyle \frac{1}{2}\mu (s) - \int_{\Tbb} K^L (s,t;\eps) \mu(t)
    \di{t}  = f(s), \quad  s\in \Tbb,
\end{equation}
with \( K^L \) given in \cref{eq:KL}.

\subsection{Analytical solution of boundary integral equation \texorpdfstring{\cref{eq:ubie_param_lap}}{\ref{eq:ubie_param_lap}}}

We compute the analytical solution of \cref{eq:ubie_param_lap} as follows.
By computing the discrete Fourier transform of \cref{eq:ubie_param_lap}, we obtain
\[
    \frac{1}{2}\mu_m - 2\pi\mathsf{K}_m \mu_{-m} = f_m, \quad
    m \in \Zbb,
\]
with \( {(\mu_m)}_{m \in \Zbb} \), \( {(\mathsf{K}_m)}_{m \in \Zbb} \), and \( {(f_m)}_{m \in \Zbb} \) denoting the complex Fourier coefficients of \( \mu \), \( K^L \), and \( f \):
\[
    X(s) \coloneqq \sum_{m \in \Zbb} X_m \ex^{\ic m s},
    \quad X = \mu,\ K^L,\ f.
\]
One can check using \cref{eq:KL} that \( K^L(s,t;\epsilon) \) boils down to a function in \( s+t \).
We rewrite this system as
\[
    \plr{\frac{1}{2} - 2\pi\mathsf{K}_0} \mu_0 = f_0
\]
and
\[
    \begin{pmatrix}
        \frac{1}{2}           & -2\pi\mathsf{K}_m \\
        - 2\pi\mathsf{K}_{-m} & \frac{1}{2}
    \end{pmatrix}
    \begin{pmatrix}
        \mu_m \\
        \mu_{-m}
    \end{pmatrix}
    =
    \begin{pmatrix}
        f_m \\
        f_{-m}
    \end{pmatrix},
    \qquad \text{for}\ m \neq 0.
\]
The solutions of these problems are given by
\begin{equation}\label{eq:system_fourier_0}
    \mu_0 = \frac{2 f_0}{1 - 4\pi\mathsf{K}_0 }
\end{equation}
and
\begin{equation}\label{eq:system_fourier_m}
    \mu_m = \frac{4}{1-16\pi^2\mathsf{K}_m\mathsf{K}_{-m}} \left(
    \frac{1}{2} f_m + 2\pi\mathsf{K}_m f_{-m} \right),
    \qquad \text{for}\ m \neq 0.
\end{equation}

Since \( K^L \) is a rational trigonometric function, one can analytically compute \( \mathsf{K}_m \) using~\cite[Eq.~(2.5)]{Geer1995}, for \( m \in \Zbb \), leading to
\[
    \mathsf{K}_m = -\frac{1}{4\pi} \rho_\eps^{|m|},
    \qquad \text{with}\ \rho_\eps = \dfrac{\eps-1}{\eps+1} \in (-1, 0).
\]
Substituting this expression for \( \mathsf{K}_{m} \) into \cref{eq:system_fourier_0} and \cref{eq:system_fourier_m}, we find that
\begin{equation}\label{eq:mu_fourier}
    \mu(s) = f_0 + \sum \limits_{m \in \Zbb^*}
    \frac{2}{1-\rho_\eps^{2|m|}} \plr{ f_m - \rho_\eps^{|m|} f_{-m} } \ex^{\ic m s},
    \quad s \in \mathbb{T} .
\end{equation}

\noindent With the solution given in \cref{eq:mu_fourier}, we find the following.
\begin{itemize}
    \item When \( f = C \) where \( C \) is a constant, \( \mu = C \).

    \item When \( f(s) = \cos(m s) \) for \( m \in \Nbb \), we have \( \mu(s)  = 2 \cos( m s) / ( 1 + \rho_{\eps}^{m} ) \).

    \item When \( f(s)  = \sin( m s ) \) for \( m \in \Nbb^* \), we have \( \mu(s)  = 2 \sin(m s) / ( 1 - \rho_{\eps}^{m} ) \).
\end{itemize}
Note that \( \rho_\eps < 0 \), and \( \rho_\eps = -1 + 2\eps + \OO(\eps^2) \) as \( \eps \to 0^+ \).
As a consequence depending on the parity of \( m \), as \( \eps \to 0^+ \), we have
\[
    1 + \rho_\eps^{m} \sim \begin{cases}
        2m\eps & \text{when \( m \) is odd} ,  \\
        1      & \text{when \( m \) is even} , \\
    \end{cases}
\]
and
\[
    1 - \rho_\eps^{m} \sim \begin{cases}
        1      & \text{when \( m \) is odd} ,  \\
        2m\eps & \text{when \( m \) is even} . \\
    \end{cases}
\]
In light of these results, we are motivated to introduce the following definition.

\begin{definition}\label{def:parity}
    For any function \( f : \Tbb \to \Cbb \):
    \begin{itemize}
        \item We say that \( f \) is \emph{even} if it belongs to the space

              \( \Cscr_\eve(\Tbb) \coloneqq \setst{f \in \Cscr(\Tbb)}{f(\pi-s) = f(s)} \);

        \item We say that \( f \) is \emph{odd} if it belongs to the space

              \( \Cscr_\odd(\Tbb) \coloneqq \setst{f \in \Cscr(\Tbb)}{f(\pi-s) = - f(s)} \).
    \end{itemize}
    We define the even and odd parts of \( f \) as \( f_\eve(s) = \frac{1}{2} (f(s) + f(\pi-s)) \) and \( f_\odd(s) = \frac{1}{2} (f(s) - f(\pi-s)) \).
\end{definition}

\noindent With \cref{def:parity} established, we state the following results.
\begin{itemize}
    \item When \( f \) is \emph{even} in the sense of \cref{def:parity} (typically \( s \mapsto \sin((2p+1) s) \) or \( \cos (2 p s) \) for \( p \in \Nbb \)), then the solution \( \mu \) is bounded as \( \eps \to 0^+ \).

    \item When \( f \) is \emph{odd} in the sense of \cref{def:parity} (typically \( s \mapsto \sin(2 p s) \) or \( \cos ((2p+1) s) \) for \( p \in \Nbb^* \)) then the solution \( \mu \) is unbounded as \( \eps \to 0^+ \).
\end{itemize}
The qualitative differences between solutions when the Dirichlet boundary data is even or odd will play an important role in addressing the nearly singular behavior of \( K^L \).

\subsection{Modified trapezoid rule (MTR)}%
\label{sec:nearly-singular-laplace}

A method to compute the leading behavior of nearly singular integrals is given in~\cite{Carvalho2020} (following classic perturbation techniques~\cite{Hinch1991}).
In what follows we apply that method to \cref{eq:ubie_param_lap}.
Recall that \( K^{L} \) exhibits nearly singular behavior on \( s + t \equiv \pi[2\pi] \).
For this reason, we introduce
\[
    I_{\delta}(\eps) = \int_{\pi - s - \frac{\delta}{2}}^{\pi - s + \frac{\delta}{2}} K^L(s,t;\eps) \, \mu(t) \di{t},
\]
where \( \delta >0 \) is fixed.
To determine the leading behavior of \( I_{\delta}(\epsilon) \) as \( \eps \to 0^{+} \), we perform a series of substitutions.
First, we shift by substituting \( t = \pi - s + x \), then we rescale using the stretched coordinate \( x = \eps X \).
These lead to
\begin{equation}\label{eq:rescaled_I_lapl}
    I_{\delta}(\eps) = \frac{1}{2\pi}
    \int_{-\frac{\delta}{2\eps}}^{\frac{\delta}{2\eps}}
    \frac{-\eps^{2}}{1+\eps^2 - (1-\eps^2)\cos(\eps X)} \, \mu(\pi - s +
    \eps X)  \di{X}.
\end{equation}
Assuming \( \mu \) is also differentiable, we expand the integrand in \cref{eq:rescaled_I_lapl} about \( \eps = 0^+ \) and find
\begin{equation}\label{eq:expansion_lap}
    \frac{-\eps^{2} \mu(\pi - s + \eps X)}{1+\eps^2 -
        (1-\eps^2)\cos(\eps X)} = -\frac{2}{4 + X^{2}} \left[ \mu(\pi - s)
        + \eps X \mu'(\pi - s) \right] + \OO(\eps^{2}).
\end{equation}
All expansions have been computed via Mathematica and SymPy, and notebooks can be found in the Github repository~\cite{PaperCode2021}.
Integrating this expansion term-by-term, the \( \OO(\eps) \) in \cref{eq:expansion_lap} vanishes upon integration since it is odd (in the classic sense), and we find
\begin{equation}\label{eq:Asympt}
    I_{\delta}(\eps) = -\frac{1}{\pi}
    \arctan\left(\frac{\delta}{4\eps} \right) \mu(\pi-s) + \OO(\eps).
\end{equation}
We refer the reader to~\cite{Carvalho2020} for a proper justification of the obtained order after integration. This result gives the asymptotic behavior of the portion of boundary integral operator in \cref{eq:ubie_param_lap} about \( s + t \equiv \pi \, [2 \pi] \) leading to its nearly singular behavior.
We now use this leading behavior to modify the Periodic Trapezoid Rule (PTR).
Given \( 2N \) quadrature points, let \( s_i = t_i = i \Delta t - \pi/2 \), for  \( i \in \llbracket 0,  2N-1 \rrbracket \) (the set of integers between \( 0 \) and \( 2N-1 \)), and with \( \Delta t = \pi/N \).
It follows that \( s_i + t_{-i \, [2N]} \equiv \pi \, [2\pi] \) (with the notation \( t_{k [N]} \) denoting \( t_k \) for \( k \in \Zbb \) modulo \( N \)).
Instead of using the trapezoid rule on those mirror points, we replace it with \( I_{\Delta t}(\eps) \) (namely \cref{eq:Asympt} where we substitute \( \delta = \Delta t \)).
We call this method the Modified Trapezoid Rule (MTR).
Applying this MTR to \cref{eq:bieLap} yields the following linear system,
\begin{equation}\label{eq:MTR-laplace}
    \frac{1}{2} \mu_i
    + \frac{1}{\pi} \arctan\plr{\frac{\Delta t}{4\eps}} \mu_{-i \, [2N]}
    - \Delta t \sum_{\substack{j = 0 \\ j \neq -i \, [2N]}}^{2N-1} K^L(s_i, t_j; \eps) \, \mu_j = f(s_i)
\end{equation}
for \( i \in \llbracket 0,  2N-1 \rrbracket \).
Using \cref{eq:MTR-laplace}, we relieve the trapezoid rule from having to approximate the underlying nearly singular behavior coming from \( K^L \).
It is in this way that we expect this modification to help in computing solutions of \cref{eq:ubie_param_lap}.

\Cref{fig:4} shows results for the relative error made by the MTR to solve \cref{eq:ubie_param_lap} using \( 2N = 64 \) quadrature points, and two different sources: \( f(s) = \cos(4 s) \), and \( f(s) = \sin(5s) \).
For comparison, we include the relative error made by the PTR and by the Quadrature by Parity Asymptotic eXpansions (QPAX) described below, with the same number of quadrature points.
Results show that the MTR works well when \( f \) is even (in the sense of \cref{def:parity}) where the relative error exhibits an \( \OO(\eps) \) behavior as \( \eps \to 0^+ \) (following the expected order from \cref{eq:Asympt}).
However when \( f \) is odd, the relative error is large and comparable to PTR for all values of \( \eps \).
Thus, the MTR is only effective for solving \cref{eq:ubie_param_lap} when \( f \) is even.

\subsection{Quadrature by parity asymptotic expansions (QPAX)}%
\label{ssec:asympt_parity_laplace}

It is clear from previous results that we must explicitly take into account the qualitative differences between even and odd boundary data to fully address the nearly singular behavior of \( K^L \).
In what follows, we develop a numerical method based on the asymptotic expansions for even and odd parities.
We call this new method Quadrature by Parity Asymptotic eXpansions (QPAX).

We now rewrite \cref{eq:ubie_param_lap} as \( \Lscr[\mu](s) = f(s) \), with \( \Lscr[\mu](s) = \left( \frac{1}{2}\Irm - \Kscr^L\right)[\mu](s) \), and
\begin{equation}\label{eq:int_op_Kl}
    \Kscr^L[\mu](s) \coloneqq \int_\Tbb K^L (s,t; \eps) \mu(t) \di{t},
    \quad s \in \Tbb.
\end{equation}
We write a formal asymptotic expansion of the operator \( \Lscr \sim \sum_{q \in \Nbb} \eps^q\, \Lscr_q \) as \( \eps \to 0^{+} \).
Using standard matched asymptotic techniques, one obtains the following result (the proof is given in \cref{ssec:KL_asy})~\cite{Carvalho2020}:

\begin{lemma}\label{lem:asy_KL}
    The integral operator \( \Kscr^L \) admits the following expansion \( \Kscr^L[\mu] = \Kscr^L_0[\mu] + \eps \Kscr^L_1[\mu] + \oo(\eps) \), for \( \mu \in \Cscr^2(\Tbb) \), where
    \[
        \Kscr^L_0[\mu](s) = -\frac{1}{2}\mu(\pi-s)
        \quad \text{and} \quad
        \Kscr^L_1[\mu](s) = -\frac{1}{2\pi} \int_\Tbb \mathsf{E}_{\pi-s}[\mu](t) \di{t}
    \]
    with \( (s, t) \mapsto \mathsf{E}_s[\mu](t) \in \Cscr(\Tbb^2) \) defined by
    \begin{equation}\label{eq:Es}
        \mathsf{E}_s[\mu](t) = \begin{dcases}
            \frac{\mu(s+t) - 2\mu(s) + \mu(s-t)}{2(1-\cos(t))} &
            \text{if } t \neq 0 \, [2\pi]                        \\
            \mu''(s)                                           &
            \text{if } t = 0 \, [2\pi]
        \end{dcases}.
    \end{equation}
\end{lemma}
\Cref{lem:asy_KL} requires at least \( \Cscr^2 \) regularity for \( \mu \) to define the continuous function \( \mathsf{E}_s \) in \cref{eq:Es} (which we assume throughout the rest of the paper). Using \cref{lem:asy_KL}, we obtain
\begin{equation}\label{eq:opL_01}
    \Lscr_0[\mu](s) = \frac{\mu(s) + \mu(\pi-s)}{2}
    \quad \text{and} \quad
    \Lscr_1[\mu](s) = \frac{1}{2\pi} \int_\Tbb \mathsf{E}_{\pi-s}[\mu](t) \di{t}.
\end{equation}
From \cref{eq:opL_01}, we directly have \( \Lscr_0[\mu] = \mu_\eve \)
and \( \Cscr_\odd(\Tbb) \subset \ker(\Lscr_0) \), therefore
\( \Lscr_0[\mu] = f \) is ill-posed for \( f \in \Cscr(\Tbb) \).  If
we assume that the right-hand side
\( f = \sum_{q \in \Nbb} \eps^q f_q \) and the density
\( \mu = \sum_{q \in \Nbb} \eps^q \mu_q \) have a smooth expansion, we
will only be able to solve for the even part.  For the odd part, the
solution is not smooth as \( \eps \to 0^+ \) and we need to consider a
different asymptotic expansion.
The need for a different asymptotic expansion for the odd part
is also revealed by the \( \OO(\eps^{-1}) \) behavior of the analytical
solution given in \cref{eq:mu_fourier}.
To take this limit behavior into account, we use the following two asymptotic expansions for the even and odd parts of \( \mu = \mu_\eve + \mu_\odd \):
\[
    \mu_\eve = \sum_{q \in \Nbb} \eps^q\, \mu_q^\eve
    \quad \text{and} \quad
    \mu_\odd = \frac{1}{\eps}\, \mu_{-1}^\odd + \sum_{q \in \Nbb} \eps^q\, \mu_q^\odd.
\]
One can show that the operator \( \Lscr \) preserves even and odd parity: \( {(\Lscr[\mu])}_{\eve / \odd} = \Lscr[\mu_{\eve / \odd}] \).
Using this result, we then substitute these expansions into \( \Lscr[\mu_\eve + \mu_\odd] = \sum_{q \in \Nbb} \eps^q\, (f_q^\eve + f_q^\odd) \), and obtain to leading order
\begin{equation}\label{eq:Lap_eq_01}
    \Lscr_0[\mu_0^\eve] = f_0^\eve,
    \qquad
    \Lscr_0[\mu_1^\eve] + \Lscr_1[\mu_0^\eve] = f_1^\eve,
    \quad \text{and} \quad
    \Lscr_1[\mu_{-1}^\odd] = f_0^\odd.
\end{equation}
The operator \( \Lscr_0 \) is invertible on \( \Cscr_\eve(\Tbb) \) (it is the identity operator).
The nullspace of \( \Lscr_1 \) is simply the set of constant functions (see \cref{lem:L1_spec}).
Therefore \( \Lscr_{1} \) is invertible on \( \Cscr_\odd(\Tbb) \).
Using \cref{eq:Lap_eq_01}, we get
\[
    \mu_0^\eve = f_0^\eve,
    \qquad
    \mu_1^\eve = f_1^\eve - \Lscr_1[f_0^\eve],
    \quad \text{and} \quad
    \mu_{-1}^\odd = {(\Lscr_1)}^{-1} [f_0^\odd].
\]
In practice we compute the approximation \( \mu \simeq \mu^\asy =  \mu_\eve^\asy +  \mu_\odd^\asy \) with
\begin{equation}\label{eq:system_parity_solve}
    \mu_\eve^\asy = f_\eve - \eps \, \Lscr_1[f_\eve]
    \quad \text{and} \quad
    \mu_\odd^\asy = \frac{1}{\eps} \, {(\Lscr_1)}^{-1} [f_\odd].
\end{equation}
Based on the chosen ansatz, we anticipate the following relative errors:
\[
    \frac{\norm{\mu_\eve -
            \mu_\eve^\asy}_{\Lrm^\infty(\Tbb)}}{\norm{\mu_\eve}_{\Lrm^\infty(\Tbb)}}
    = \OO\plr{\eps^2} \quad \text{and} \quad \frac{\norm{\mu_\odd
            -
            \mu_\odd^\asy}_{\Lrm^\infty(\Tbb)}}{\norm{\mu_\odd}_{\Lrm^\infty(\Tbb)}}
    = \OO\plr{\eps}.
\]

\begin{remark}\label{rem:eigenvalue}
    The decomposition into even and odd parts can be also understood from a spectral point of view. For the case of the high aspect ratio ellipse, it is known that \( \Kscr^L \) exhibits eigenvalues \( {(\lambda_n^{\eve/ \odd})}_n \) (associated to even eigenfunctions, odd eigenfunctions, respectively) such that \( \lambda_n^{\eve}  \underset{n \to \infty}{\longrightarrow}- \frac{1}{2} \), \( \lambda_n^{\odd}  \underset{n \to \infty}{\longrightarrow} \frac{1}{2} \) (see~\cite{ando2020spectral} for its generalization to two dimensional thin planar domains such as rectangles). It is also the reason why, when considering \( \Lscr^L = \frac{1}{2}\Irm - \Kscr^L \), ill-posedness arises in presence of an odd source term.
\end{remark}

We now turn to the discretization, for which we will need to separate even and odd parts.
From \cref{lem:L1_spec}, we find that for any \( m \in \Zbb \),
\[
    \Lscr_1\left[ \cos\plr{ m t } \right](s) = {(-1)}^{m+1} m \cos\plr{ m s },
    \quad
    \Lscr_1\left[ \sin\plr{ m t } \right](s) = {(-1)}^m m \sin\plr{ m s }.
\]
Using these two relations, we define the discretized operators \( \Lbb_1^\eve \) and \( \Lbb_1^\odd \) as follows. Given \( 2N \) uniformly spaced grid points starting at \( - \frac{\pi}{2} \) (bottom of the narrow ellipse), one can simply use the half grid to define the even and odd parts.
Let \( \mathsf{C}_{N+1} \) denote the \( (N+1) \times (N+1) \) matrix corresponding to the discrete cosine transform with the following entries,
\begin{align*}
    \plr{\mathsf{C}_{N+1}}_{i,0} = \sqrt{\frac{2}{N}} \frac{1}{2},
     &  &
    \plr{\mathsf{C}_{N+1}}_{i,j} = \sqrt{\frac{2}{N}} \cos\plr{\frac{i j \pi}{N}},
     &  &
    \plr{\mathsf{C}_{N+1}}_{i,N} = \sqrt{\frac{2}{N}} \frac{{(-1)}^{i}}{2},
\end{align*}
for \( (i, j) \in \llbracket 0, N \rrbracket^2 \).
In terms of \( \mathsf{C}_{N+1} \), we define \( \Lbb_1^\eve \) according to
\begin{equation}
    \Lbb_1^\eve = \mathsf{C}_{N+1} \, \diag\plr{0, 1, \ldots, N} \, \mathsf{C}_{N+1}.
    \label{eq:L1-even}
\end{equation}
Let \( \mathsf{S}_{N-1} \) denote the \( (N-1) \times (N-1) \) matrix corresponding to the discrete sine transform with entries
\begin{equation*}
    \plr{\mathsf{S}_{N-1}}_{i,j}= \sqrt{\frac{2}{N}} \sin\plr{\frac{i j
            \pi}{N}},
\end{equation*}
for \( (i, j) \in \llbracket 1, N-1 \rrbracket^2 \). In terms of \( \mathsf{S}_{N-1} \), we
define \( \Lbb_1^\odd \) according to
\begin{equation}
    \Lbb_1^\odd  = \mathsf{S}_{N-1} \, \diag\plr{-1, -2, \ldots, -(N-1)} \, \mathsf{S}_{N-1} .
    \label{eq:L1-odd}
\end{equation}

We now use these asymptotic results to develop a numerical method to
solve \cref{eq:ubie_param_lap}.
Given the Dirichlet boundary data \( f \), we first compute the vectors,
\begin{equation*}
    \vect{f}_{N+1}^\eve
    = \plr{ f_\eve(s_i) }_{i \in \llbracket 0, N \rrbracket}
    = \plr{ \frac{1}{2} f(s_i) + \frac{1}{2} f(\pi - s_{i}) }_{i\in \llbracket 0, N \rrbracket},
\end{equation*}
and
\begin{equation*}
    \vect{f}_{N-1}^\odd
    = \plr{ f_\odd(s_i) }_{i \in \llbracket 1, N-1 \rrbracket}
    = \plr{ \frac{1}{2} f(s_i) - \frac{1}{2} f(\pi - s_{i}) }_{i \in \llbracket 1, N-1 \rrbracket},
\end{equation*}
where \( s_i = i \Delta t - \frac{\pi}{2} \), \( i \in \llbracket 0, 2N-1 \rrbracket \) are the PTR quadrature points.
Next, we compute the numerical approximation of \cref{eq:system_parity_solve} through evaluation of
\begin{equation*}
    \vect{\mu}_{N+1}^\eve = \vect{f}_{N+1}^\eve - \eps \, \Lbb_1^\eve \vect{f}_{N+1}^\eve
    \quad \text{and} \quad
    \vect{\mu}_{N-1}^\odd = \eps^{-1} \, \plr{\Lbb_1^\odd}^{-1} \vect{f}_{N-1}^\odd.
\end{equation*}
With these results, we compute the approximation
\begin{equation}\label{eq:eve_odd_to_mu}
    \mu(s_{i}) \approx \begin{cases}
        \plr{ \vect{\mu}_{N+1}^\eve }_i
         & i = 0, N ,
        \\[1ex]
        \plr{ \vect{\mu}_{N+1}^\eve }_i + \plr{ \vect{\mu}_{N-1}^\odd }_i
         & i \in \llbracket 1,  N-1 \rrbracket,
        \\[1ex]
        \plr{ \vect{\mu}_{N+1}^\eve }_{2N-i} + \plr{ \vect{\mu}_{N-1}^\odd }_{2N-i}
         & i \in \llbracket N+1,  2N-1 \rrbracket.
    \end{cases}
\end{equation}
We denote \( \vect{\mu}_{2N} \) to be the \( 2N \) vector whose entries are given by \cref{eq:eve_odd_to_mu}.

In \cref{fig:4} we show the relative error, \( \norm{\vect{\mu}_{2N} - \vect{\mu}_\ana}_\infty / \norm{\vect{\mu}_\ana}_\infty \) (with \( \vect{\mu}_\ana \) the vector denoting discrete analytic solution at the quadrature points computed via \cref{eq:mu_fourier}) as a function of \( \eps \) using PTR,
MTR, and QPAX to compute the solution of \cref{eq:ubie_param_lap}. All codes are publicly available on Github~\cite{PaperCode2021}.
For these results, the number of quadrature points for all methods is \( 2N = 64 \).
The results in the left plot are for the even source \( f(s) = \cos(4 s) \) and the results in the right plot are for the odd source \( f(s) = \sin(5s) \).
Results indicate, as expected, that the error made by QPAX is \( \OO(\eps^{2}) \) for an even source and \( \OO(\eps) \) for an odd source.
For the considered even sources, we have taken the two-term asymptotic approximation in \cref{eq:system_parity_solve} which leads to the \( \OO(\eps^3) \) because the second-order term vanishes. From the parity asymptotic expansions, we found that the odd part contributes globally to the nearly singular behavior. It is the reason why the MTR (based only on local inner expansion) fails for odd sources.

\begin{figure}[hbt]
    \centering
    \begin{subfigure}{0.49\textwidth}
        \centering
        \includegraphics{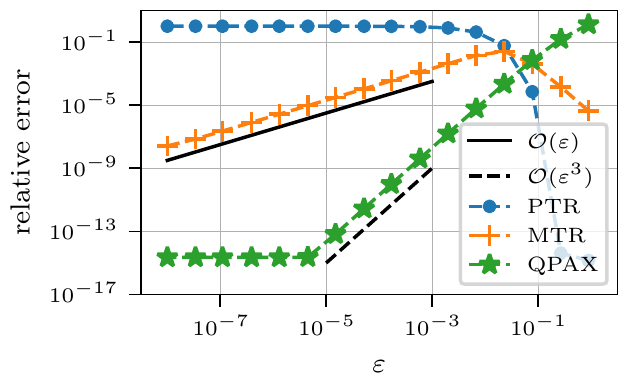}
        \caption{Even: \(f(s) = \cos(4 s)\)}
    \end{subfigure}
    \begin{subfigure}{0.49\textwidth}
        \centering
        \includegraphics{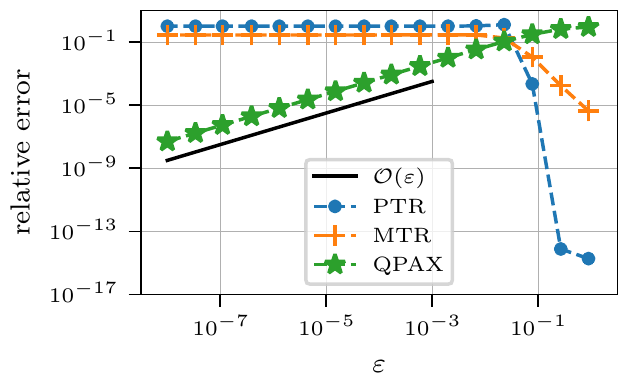}
        \caption{Odd: \(f(s) = \sin(5 s)\)}
    \end{subfigure}
    \caption{%
        Results for relative error as a function of \( \eps \) in the     numerical solution of the boundary integral equation for the Dirichlet problem for Laplace's equation in a high aspect ratio ellipse for PTR (blue curves with ``\( \bullet \)'' symbols), MTR (orange curves with ``+'' symbols), and QPAX (green curves with ``\( \star \)'' symbols).
        All of these results were computed using \( 2N = 64 \) quadrature points.
        In the left plot, we show relative error for the even source \( f(s) = \cos(4 s) \) and in the right plot, we show relative error results for the odd source \( f(s) = \sin(5s) \).
        The solid black line represents the \( \OO(\eps) \) convergence slope and the dashed black line represents the \( \OO(\eps^3) \) convergence slope.
    }%
    \label{fig:4}
\end{figure}

\begin{remark}
    For simplicity, we have explicitly used a spectral decomposition to discretize the operator \( \Lscr_1 \).
    This is obtained using \cref{lem:L1_spec}.
    Other discretizations are possible as long as we split the expansion as in \cref{eq:system_parity_solve}.
    For example Chebyshev nodes on interval \( \plr{ -\frac{\pi}{2}, \frac{\pi}{2} } \) may be appropriate since the operator \( \Lscr_1 \) is diagonalizable on \( \Cscr_{\eve / \odd}^2(\Tbb) \).
\end{remark}

\section{QPAX for scattering by a high aspect ratio ellipse}%
\label{sec:QPAX}

From the discussion on the Dirichlet problem for Laplace's equation in a high aspect ratio ellipse, we found that it is necessary to consider different asymptotic expansions for the even and odd parts of the solution.
We now extend those results to the scattering problem by a sound-hard, high aspect ratio ellipse.

We rewrite \cref{eq:ubie} on the high aspect ratio ellipse as \( \Lscr^H[u] = u^\inc \) with \( \Lscr^H \coloneqq \frac{1}{2}\Irm - \Kscr^H \) and the operator \( \Kscr^H[\mu](s) = \int_{\Tbb} K(s,t; \eps) \mu(t) \di{t} \).
Using \cref{eq:kernel-soundhard}, we write the kernel of \( \Kscr^H \) as \( K(s,t; \eps) = H(z_\eps(s, t)) \, K^L(s,t; \eps) \) where
\begin{equation}\label{eq:zeps}
    z_\eps(s, t) \coloneqq k \, r(s,t; \eps) = 2\, k\,
    \abs{ \sin\left(\tfrac{s-t}{2}\right) } \,
    \sqrt{ {\cos\plr{ \tfrac{s+t}{2} }}^2 + \eps^2 \sin\plr{ \tfrac{s+t}{2} }^2 },
\end{equation}
and \( H \) denotes the continuous function \( H(z) \coloneqq \frac{\ic \pi}{2} \, z \Ho_1(z) \).
Although \( H \) is continuous, its derivative has a logarithmic singularity at \( s = t \) (i.e. \( z = 0 \)).
To address this singularity in the derivative, we use~\cite[Section.~10.8]{Nist} to decompose \( H \) according to
\[
    H(z) = \Psi(z) - \frac{z \bJ_1(z)}{2} \ln\left(\frac{4
        z^2}{k^2}\right),
\]
where \( \Psi \) is an analytic function on \( \Cbb \) satisfying \( \Psi(0) = 1 \) and \( \Psi'(0) = 0 \).
Then we split the integral operator according to \( \Kscr^H = \Kscr^\Psi + \Kscr^{\ln} \) where
\begin{subequations}
    \begin{align}
        \Kscr^\Psi[\mu](s)
         & = \int_\Tbb K^L(s, t; \eps) \ \Psi(z_\eps(s, t)) \ \mu(t) \di{t} \label{eq:KPsi},
        \\
        \Kscr^{\ln}[\mu](s)
         & = -\frac{1}{2} \int_\Tbb K^L(s, t; \eps) \ z_\eps(s, t) \bJ_1(z_\eps(s, t)) \ln\left(\frac{4 {z_\eps(s, t)}^2}{k^2}\right) \ \mu(t) \di{t} \label{eq:Kln}.
    \end{align}
\end{subequations}

We now state two useful Lemmas.

\begin{lemma}\label{lem:asy_KPsi}
    The integral operator \( \Kscr^\Psi \) defined in \cref{eq:KPsi} admits the expansion \( \Kscr^\Psi[\mu] = \Kscr^\Psi_0[\mu] + \eps \Kscr^\Psi_1[\mu] + \oo(\eps) \), for \( \mu \in \Cscr^2(\Tbb) \), where
    \[
        \Kscr^\Psi_0[\mu](s) = -\frac{1}{2}\mu(\pi-s)
        \quad \text{and} \quad
        \Kscr^\Psi_1[\mu](s) = -\frac{1}{2\pi} \int_\Tbb \mathsf{E}_{\pi-s}[\psi_s \, \mu](t) \di{t}
    \]
    with \( \psi_s(t) = \Psi\plr{ 2\, k\, \abs{ \sin\left(\tfrac{s-t}{2}\right) \cos\left(\tfrac{s+t}{2}\right) } } \), and \( \mathsf{E}_{s} \) defined in \cref{eq:Es}.
\end{lemma}

\begin{lemma}\label{lem:asy_Kln}
    The integral operator \( \Kscr^{\ln} \) defined in \cref{eq:Kln} admits the expansion
    \( \Kscr^{\ln}[\mu] = \eps \Kscr^{\ln}_1[\mu] + \oo(\eps) \), for
    \( \mu \in \Cscr^2(\Tbb) \), where
    \[
        \Kscr^{\ln}_1[\mu](s) = \frac{1}{2\pi} \int_\Tbb {[\phi_s \,
                    \mu]}_\eve(t) \ \ln\left(4 \, {\sin\left(\tfrac{s-t}{2}\right)}^2
        \right) \di{t}
    \]
    with
    \begin{equation}\label{eq:Phi_expression}
        \phi_s(t) = \begin{dcases}
            k \abs{\frac{\sin\plr{\tfrac{s-t}{2}}}{\cos\plr{\tfrac{s+t}{2}}}}
            \bJ_1\plr{ 2\, k\, \abs{ \sin\plr{\tfrac{s-t}{2}} \cos\plr{\tfrac{s+t}{2}} } }
             & \text{if } t \not \equiv \pi-s \, [2\pi] \\
            2 k^2 \sin\plr{\tfrac{s-t}{2}}^2
             & \text{if } t \equiv \pi-s \, [2\pi]
        \end{dcases}.
    \end{equation}
\end{lemma}

The proof of the expansions for \( \Kscr^{\Psi} \) and \( \Kscr^{\ln} \),
defined in \cref{eq:KPsi} and \cref{eq:Kln}, can be found in
\cref{ssec:Kpsi_asy} and \cref{ssec:say_Kln}, respectively.
Using \cref{lem:asy_KPsi} and~\cref{lem:asy_Kln} we can write the asymptotic expansion of the integral operator \( \Lscr^H = \Hscr_0 + \eps \, \Hscr_1 + \oo(\eps) \) where
\begin{subequations}\label{eq:H_asy}
    \begin{align}
        \Hscr_0[\mu](s)
         & = \frac{\mu(s) + \mu(\pi-s)}{2} = \mu_\eve(s)
        \\
        \Hscr_1[\mu](s)
         & = \frac{1}{2\pi} \int_\Tbb \mathsf{E}_{\pi-s}[\psi_s \, \mu](t) \di{t}
        - \frac{1}{2\pi} \int_\Tbb {[\phi_s \, \mu]}_\eve
        \ln\plr{ 4 {\sin\left(\tfrac{s-t}{2}\right)}^2 } \di{t}.
    \end{align}
\end{subequations}
Similar to the Dirichlet problem for the Laplace's equation, we have \( \ker(\Hscr_0) = \Cscr_\odd(\Tbb) \) and the leading order problem \( \Hscr_0 [u] = u^\inc \) is not well-posed on \( \Cscr(\Tbb) \).
As a consequence we need to separate the even and odd terms to solve \( \Lscr^H [u] = u^\inc \).
We write the source term as \( f(s) = u^\inc(y(s)) = u^\inc(\eps \cos(s), \sin(s)) \), and we expand \( f = f_\eve + f_\odd \) with \( f_\eve= \sum_{q \in \Nbb} \eps^q \, f_q^\eve(s) \) and \( f_\odd = \sum_{q \in \Nbb} \eps^q \, f_q^\odd(s) \).
To highlight the parity scales, we write as before \( u = u_\eve + u_\odd \) with the ansatz
\begin{equation}\label{eq:ansatz_scattering}
    u_\eve = \sum_{q \in \Nbb} \eps^q\, u_q^\eve
    \quad \text{and} \quad
    u_\odd = \frac{1}{\eps}\, u_{-1}^\odd + \sum_{q \in \Nbb} \eps^q\, u_q^\odd
\end{equation}
as \( \eps \to 0^+ \).
For the high aspect ratio ellipse, the expansions of the source term give us
\begin{subequations}
    \begin{align}
        f_0^\eve(s) & = u^\inc(0, \sin(s)), & f_1^\eve(s) & = 0, \label{eq:f_ev}
        \\
        f_0^\odd(s) & = 0,                  & f_1^\odd(s) & = \cos(s)\, \partial_x u^\inc(0, \sin(s)).\label{eq:f_od}
    \end{align}
\end{subequations}
Contrary to Laplace's problem, the scattering problem is always well-posed, we know that \( u_{-1}^\odd \equiv 0 \), and \( f_0^\odd \equiv 0 \) is expected.
However ill-posedness of the asymptotic problem as \( \eps \to 0^+ \) remains (but it is subtle): note that the leading order term of \( f_\odd \) is \( \OO(\eps) \) while the one for \( u_\odd \) is \( \OO(1) \).
It is the reason why we still shift the power index of \( \eps \).
To obtain leading behavior of the solution of \( \Lscr^H[u] = u^\inc \), we replace \( \Lscr^H \) by \( \Hscr_0 + \eps \, \Hscr_1 \), substitute the expansions for \( u_{\eve} \) and \( u_{\odd} \) defined in \cref{eq:ansatz_scattering}, and use the fact that \( \Lscr^H \) (and consequently \( \Hscr_0, \Hscr_1 \)) preserves parity.
In the end we obtain similar equations as \cref{eq:Lap_eq_01}: \( \Hscr_0[u_0^\eve] = f_0^\eve \), \( \Hscr_0[u_1^\eve] = - \Hscr_1[f_0^\eve] \), and \( \Hscr_1 [u_0^\odd] = f_1^\odd \).
In practice we simply need to compute
\begin{subequations}\label{eq:system_parity_solve_scatt_2}
    \begin{align}
        u_\eve^\asy(s)
         & = u^\inc(0, \sin(s)) - \eps \, \Hscr_1 \left[ u^\inc(0, \sin(s)) \right],
        \\[1ex]
        u_\odd^\asy(s)
         & = \Hscr_1^{-1} \left[ \cos(s) \, \partial_x u^\inc(0, \sin(s)) \right] .
    \end{align}
\end{subequations}
We obtain an asymptotic approximation \( u \simeq u^\asy =  u_\eve^\asy +  u_\odd^\asy \) in the limit as \( \eps \to 0^+ \), whose error is \( \oo(\eps) \).
As for the Laplace case, we anticipate the following relative errors:
\[
    \frac{\norm{u_\eve - u_\eve^\asy}_{\Lrm^\infty(\Tbb)}}{\norm{u_\eve}_{\Lrm^\infty(\Tbb)}}
    = \OO\plr{\eps^2}
    \quad \text{and} \quad
    \frac{\norm{u_\odd - u_\odd^\asy}_{\Lrm^\infty(\Tbb)}}{\norm{u_\odd}_{\Lrm^\infty(\Tbb)}}
    = \OO\plr{\eps}.
\]
We now introduce a numerical method for computing \cref{eq:system_parity_solve_scatt_2}.
This method modifies the product Gaussian quadrature rule by Kress~\cite{Kress1991} for two-dimensional scattering problems.
Using the PTR quadrature points \( s_i = t_i = i \Delta t - \pi/2 \) for \( i\in \llbracket 0,  2N-1 \rrbracket \) with \( \Delta t = \pi/N \), Kress' Gaussian Product Quadrature Rule (PQR) is given by
\begin{equation}\label{eq:Kress-quad}
    \int_\Tbb K(s_i, t; \eps) u(t) \di{t}
    \approx \sum_{j=0}^{2N -1} \left[R_{|N-j|}^{(N)} \, K_1(s_i,t_j)
        + \Delta t \, K_2(s_i, t_j) \right] u_j,
\end{equation}
with
\[
    K_1(s, t; \eps) = -\frac{1}{2} z_\eps(s, t) \bJ_1 (z_\eps(s, t)) K^L(s,t; \eps),
\]
and
\[
    K_2(s, t; \eps) = \frac{1}{2}  z_\eps(s, t)
    \left[ \ic \pi \Ho_1(z_\eps(s, t)) + \bJ_1(z_\eps(s, t))
        \ln\plr{ 4\, {\sin\plr{ \tfrac{s-t}{2} }}^2 } \right]
    K^L(s, t;\eps).
\]
Here, \( \bJ_1 \) is the Bessel function of first kind and of order one and \( K^L \) is given in \cref{eq:KL}.
The quadrature weights in \cref{eq:Kress-quad} are given by
\begin{equation}\label{eq:Kress_weights}
    R_k^{(N)} \coloneqq - \frac{{(-1)}^k\pi}{N^2} -\frac{2\pi}{N}
    \sum_{m=1}^{N-1} \frac{1}{m}\cos \left( \frac{mk\pi}{N} \right),
    \quad k \in  \llbracket 0, N \rrbracket .
\end{equation}
Using this quadrature rule within a Nystr\"om method to solve \cref{eq:ubie}, we obtain
\begin{equation}\label{eq:PQR-system}
    \frac{1}{2} u_i - \sum_{j=0}^{2N-1} \left[ R_{|N-j|}^{(N)} K_1(s_i,
        t_j) + \frac{\pi}{N} K_2(s_i, t_j) \right] u_j = u^\inc(s_i),
    \quad i \in  \llbracket 0, 2 N -1 \rrbracket.
\end{equation}
This PQR explicitly addresses the weak singularity in the derivative in \( K \) on \( s = t \) due to the Hankel function \( \Ho_{1} \).
However, it does not address the nearly singular behavior due to \( K^L \), so we do not expect it to work well for the high aspect ratio ellipse.

\bigskip

We can easily modify the PQR method to include the modification we derived in \cref{sec:nearly-singular-laplace} based on the asymptotic expansion for \( K^L \).
First, we observe that \( K_1 = \OO(\eps) \) as \( \eps \to 0^{+} \), so it does not exhibit any nearly singular behavior.
Consequently, we do not need to modify that part of the PQR\@.
However, \( K_2 \) exhibits a nearly singular behavior as \( \eps \to 0^{+} \), so we write, for \( i \in  \llbracket 0, 2 N -1 \rrbracket \),
\[
    \int_{\Tbb} K_2(s_i, t; \eps) u(t) \di{t}
    \approx - \frac{1}{\pi} \arctan\plr{\frac{\Delta t}{4\eps}} u_{-i \, [2N]} + \Delta t \sum_{\substack{j = 0 \\ j \neq -i \, [2N]}}^{2N-1} K_2(s_i, t_j; \eps) \, u_j.
\]
Incorporating this modification into the Nystr\"om method into \cref{eq:PQR-system}, we obtain
\begin{multline}\label{eq:BIE_disc_helm}
    \frac{1}{2} u_i
    - \sum_{j=0}^{2N -1} R_{|N-j|}^{(N)} \, K_1(s_i,t_j) u_j
    + \frac{1}{\pi} \arctan\plr{\frac{\Delta t}{4\eps}} \mu_{-i \,
        [2N]}\\
    - \Delta t \sum_{\substack{j = 0 \\ j \neq -i \, [2N]}}^{2N-1}
    K_2(s_i, t_j; \eps) \, u_j
    = u^\inc(s_i),
\end{multline}
for \( i \in  \llbracket 0, 2N-1 \rrbracket \).
We call \cref{eq:BIE_disc_helm} the Modified Product Quadrature Rule (MPQR).
We expect MPQR to fail when considering an odd source term.

\bigskip

We now modify PQR further to extend QPAX to solve \cref{eq:ubie}.
Using \cref{eq:H_asy}, we define the discretized even operator \( \Hbb_1^\eve \) and the discretized odd operator \( \Hbb_1^\odd \) as follows.
First, we split \( \Hbb_1^{\eve / \odd} = \Hbb_{1, \psi}^{\eve / \odd} - \Hbb_{1,\ln}^{\eve / \odd} \) into the two parts corresponding to the two integrals in \cref{eq:H_asy}. As before, we make use of the first half of the quadrature points to create the even and odd discrete operators. The entries for \( \Hbb_{1, \psi}^\eve \) are given by
\[
    \plr{ \Hbb_{1, \psi}^\eve }_{i, j}
    = \plr{\Lbb_1^\eve}_{i,j} \, \psi_{s_i}(t_j),
    \quad (i, j) = \llbracket 0, N \rrbracket^2,
\]
and the entries of \( \Hbb_{1, \psi}^\odd \) are given by
\[
    \plr{ \Hbb_{1, \psi}^\odd }_{i, j} =
    \plr{\Lbb_1^\odd}_{i,j} \, \psi_{s_i}(t_j),
    \quad(i, j) = \llbracket 1, N-1 \rrbracket^2,
\]
with \( \Lbb_{1}^\eve \) given in \cref{eq:L1-even} and \( \Lbb_{1}^\odd \) given in \cref{eq:L1-odd}.
One can check that the entries of \( \Hbb_{1, \ln}^{\eve / \odd} \) are given by
\begin{align*}
    \plr{ \Hbb_{1, \ln}^\eve }_{i,j}
     & = \plr{W_{N+1}^\eve}_{i,j} \, \frac{\phi_{s_i}(t_j) + \phi_{s_i}(\pi-t_j)}{2},
     &
     & (i, j) = \llbracket 0, N \rrbracket^2,
    \\[
    1ex]
    \plr{ \Hbb_{1, \ln}^\odd }_{i,j}
     & = \plr{W_{N+1}^\eve}_{i,j} \, \frac{\phi_{s_i}(t_j) - \phi_{s_i}(\pi-t_j)}{2}
     &
     & (i, j) = \llbracket 1, N-1 \rrbracket^2,
\end{align*}
respectively, where \( W_{N+1}^\eve \) is the matrix of Kress weights
for integrating even functions defined by
\[
    W_{N+1}^\eve
    = \mathsf{C}_{N+1} \, \diag\plr{0, -\frac{2\pi}{1}, -\frac{2\pi}{2}, \ldots, -\frac{2\pi}{N}} \, \mathsf{C}_{N+1}
\]
with \( \mathsf{C}_{N+1} \) denoting the same matrix of the discrete
cosine transform used in \cref{eq:L1-even}.
As for the Laplace case, from \cref{eq:system_parity_solve_scatt_2}, we obtain the
approximation of the unknowns \( \vect{u}_{N+1}^\eve \approx {(u_\eve(s_i))}_{i\in \llbracket 0, N \rrbracket} \) and \( \vect{u}_{N-1}^\odd \approx {(u_\odd(s_i))}_{i\in \llbracket 1, N-1 \rrbracket} \) as
\[
    \vect{u}_{N+1}^\eve = \vect{f}_{N+1}^\eve - \eps \, \Hbb_1^\eve \vect{f}_{N+1}^\eve
    \quad \text{and} \quad
    \vect{u}_{N-1}^\odd = \plr{\Hbb_1^\odd}^{-1} \vect{f}_{N-1}^\odd
\]
where \( \vect{f}_{N+1}^\eve = \plr{u^\inc(0, \sin(s_i))}_{i\in \llbracket 0, N \rrbracket} \) and \( \vect{f}_{N-1}^\odd = \plr{\cos(s_i) \, \partial_x u^\inc(0, \sin(s_i))}_{i\in \llbracket 1, N -1\rrbracket} \).
Then, we recombine the even and odd parts using \cref{eq:eve_odd_to_mu} to obtain the approximation vector \( \vect{u}_{2N} \).

\bigskip

We now study the relative error \( \norm{\vect{u}_{2N} - \vect{u}_\ana}_\infty / \norm{\vect{u}_\ana}_\infty \) (with \( \vect{u}_\ana \) the vector denoting discrete analytic solution at the quadrature points computed via \cref{eq_u_ana}) as a function of \( \eps \) made by PQR, MQPR, and QPAX\@.
All codes are publicly available on Github~\cite{PaperCode2021}.
In \cref{fig:5}, we show these relative errors for \( u^\inc(x, y) = \Mce{1}_3(\xi, q) \ce_3(\eta, q) \) which produces a purely even source on the boundary (left plot) and \( u^\inc(x, y) = \Mse{1}_2(\xi, q) \se_2(\eta, q) \) which produces a purely odd source on the boundary (right plot).
For these results, the number of quadrature points is fixed at \( 2N = 64 \) for all of the approximations.
These results mimic what we had seen for the Dirichlet problem for Laplace's equation in a high aspect ratio ellipse.
Here, the PQR approximation produces large errors as \( \eps \to 0^+ \) because it does not account for the nearly singular behavior of the integral operator for a high aspect ratio ellipse.
For an even source, the relative error of the MPQR exhibits an \( \OO(\eps) \), but it does not work well for an odd source.
In contrast, the relative error of the QPAX approximation exhibits, as anticipated, an \( \OO(\eps^{2}) \) behavior for the even source and an \( \OO(\eps) \) behavior for the odd source.
Hence, the QPAX approximation effectively addresses the inherent parity in the nearly singular behavior associated with this high aspect ratio ellipse.

%
%
\begin{figure}[hbt]
    \centering
    \begin{subfigure}{0.49\textwidth}
        \centering
        \includegraphics{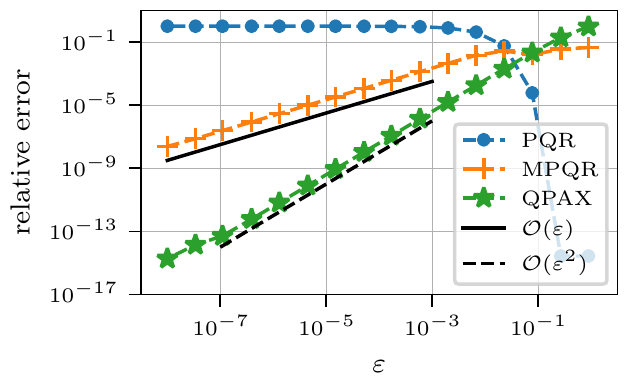}
        \caption{Even source: \(u^\inc(x, y) = \Mce{1}_3(\xi, q) \ce_3(\eta, q)\)}
    \end{subfigure}
    \begin{subfigure}{0.49\textwidth}
        \centering
        \includegraphics{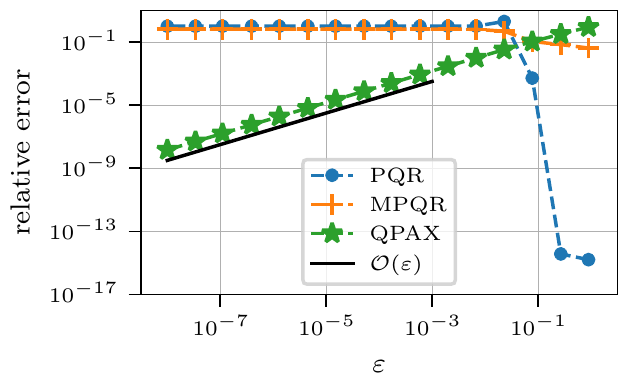}
        \caption{Odd source: \(u^\inc(x, y) = \Mse{1}_2(\xi, q) \se_2(\eta, q)\)}
    \end{subfigure}
    \caption{%
        Results for relative error as a function of \( \eps \) in the numerical solution of the boundary integral equation for scattering by a sound-hard, high aspect ratio ellipse for PQR (blue curves with ``\( \bullet \)'' symbols), MPQR (orange curves with ``+'' symbols), and QPAX (green curves with ``\( \star \)'' symbols).
        All of these results were computed using \( 2N = 64 \) quadrature points. The left plot shows results for an incident field that produces an even source on the boundary and the right plot shows results for an incident field that produces an odd source on the boundary.
        The black line represents the \( \OO(\eps) \) convergence slope and the dashed black line represents the \( \OO(\eps^2) \) convergence slope.
    }%
    \label{fig:5}
\end{figure}

Because the relative error for the QPAX approximation for odd parity exhibits an \( \OO(\eps) \) behavior, we expect that it exhibits an \( \OO(\eps) \) behavior for a general incident field.
To verify that this is indeed true, we show in \cref{fig:7a} the relative error with respect to the infinity norm as a function of \( \eps \) for the approximate plane wave given by \cref{eq:u_inc_mat}.
This incident field has both even and odd components.
Again we see that PQR and MPQR do not work well for this problem. In contrast, the relative error exhibited by QPAX is \( \OO(\eps) \) as expected.

%
%
\begin{figure}[!hbt]
    \centering
    \begin{subfigure}{0.49\textwidth}
        \centering
        \includegraphics{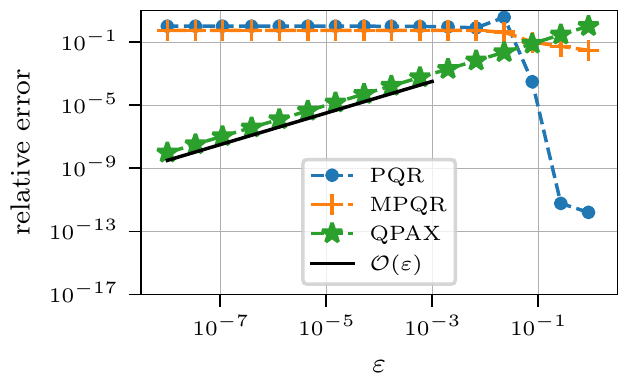}
        \caption{Approximate plane wave \cref{eq:u_inc_mat}}%
        \label{fig:7a}
    \end{subfigure}
    \begin{subfigure}{0.49\textwidth}
        \centering
        \includegraphics{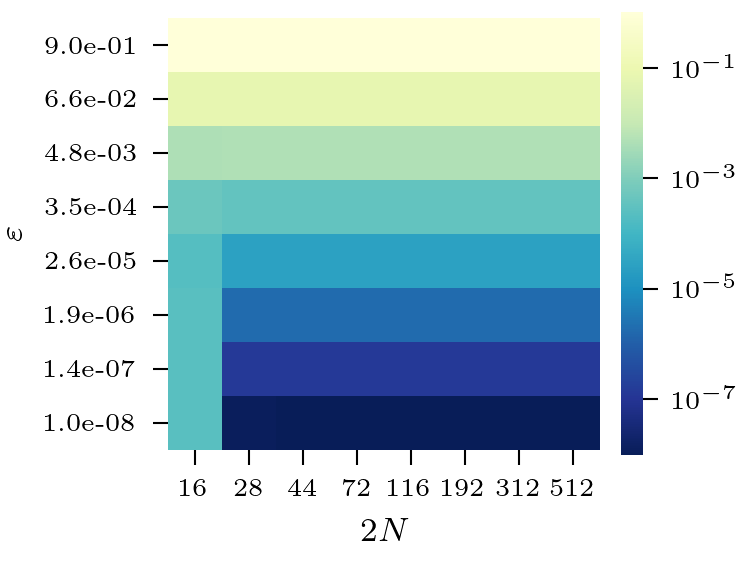}
        \caption{Relative error for QPAX}%
        \label{fig:7b}
    \end{subfigure}
    \caption{%
        (Left) Relative error as in \cref{fig:5}, but for the approximate plane wave incident field given in \cref{eq:u_inc_mat}.
        (Right) The relative error for QPAX (log scale) plotted as a function of the number of quadrature points, \( 2N \), and the high aspect ratio parameter, \( \eps \), for the approximate plane wave incident field given in \cref{eq:u_inc_mat}.}%
    \label{fig:7}
\end{figure}

It appears that the QPAX approximation is effective for studying scattering by a high aspect ratio ellipse.
Thus, we evaluate its relative error with respect to the infinity norm as a function of \( \eps \) and \( 2N \) in \cref{fig:7b}.
Results show the effectiveness of the QPAX approximation over a range of aspect ratios and computational resolutions.
The relative error for \( 2N = 16 \) appears to saturate as \( \eps \to 0^+ \).
For that case, the fields on the boundary are underresolved with that number of quadrature points leading to a dominating aliasing error.
For all other cases, we observe that the relative error behaves as \( \OO(\eps) \) as \( \eps \to 0^+ \).
Results in \cref{fig:8} show that the performance of the QPAX method does not depend on the direction \( d = (\cos(\alpha), \sin(\alpha)) \) of the incident field. We consider \( u^\inc \) defined in \cref{eq:u_inc_mat} where we choose \( \alpha_m^\inc = 2\, \ic^m\, \ce_m(\frac{\pi}{2} -\alpha, q) \) and \( \beta_m^\inc = 2\, \ic^m\, \se_m(\frac{\pi}{2} - \alpha, q) \), \( \alpha \in [0, \frac{\pi}{2}] \). Note that the choice \( \alpha = \frac{\pi}{2} \) is a special case: the incident field is even with respect to the major axis of the ellipse, therefore the odd part of the solution vanishes, leading to an \( \OO(\eps^2) \) relative error.
Thus, the QPAX approximation is highly effective (accurate) and efficient (requiring modest resolution) for scattering problems by a high aspect ratio ellipse.

%
%
\begin{figure}[!hbt]
    \centering
    \begin{subfigure}{0.49\textwidth}
        \centering
        \includegraphics{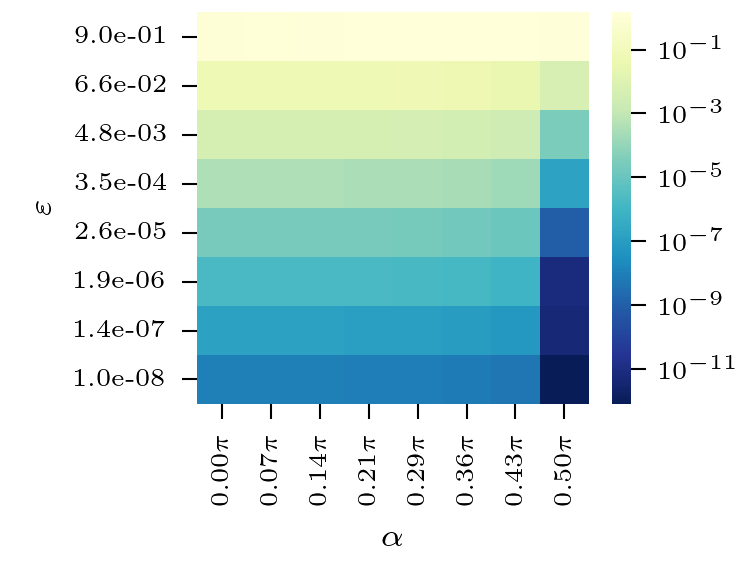}
        \caption{Relative error for QPAX with respect to \( \alpha \), \( \eps \).}
    \end{subfigure}
    \begin{subfigure}{0.49\textwidth}
        \centering
        \includegraphics{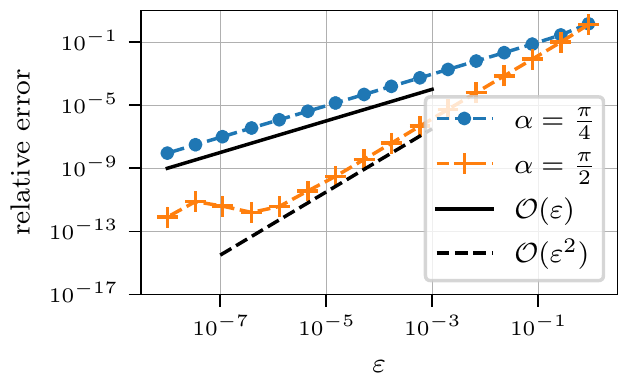}
        \caption{Plane waves of incidence \( \alpha = \left\lbrace \frac{\pi}{4}, \frac{\pi}{2} \right\rbrace \)}
    \end{subfigure}
    \caption{%
        (Left) Plot of the relative error for the QPAX method with respect to \( 0 \le \alpha \le \frac{\pi}{2} \), \( 10^{-8} \le \eps \le 0.9 \), and \( 2N = 64 \).
        (Right) Plot of the relative error for the QPAX method with respect to \( \alpha = \left\lbrace \frac{\pi}{4}, \frac{\pi}{2} \right\rbrace \), \( 10^{-8} \le \eps \le 0.9 \), and \( 2N = 64 \).}%
    \label{fig:8}
\end{figure}

\section{Extensions}\label{sec:extensions}

The results discussed above highlight the crucial role of parity for
scattering by a high aspect ratio sound-hard ellipse. We show here
that parity is important for other related problems including the
scattering problem for sound-soft high aspect ratio ellipse, the
transmission problem for a penetrable high aspect ratio ellipse, and
scattering problems for more general high aspect ratio
particles. For each of these problems, we identify their key
features and find that they have been addressed in our discussion of
the sound-hard problem.
\subsection{Sound-soft high aspect ratio ellipse}
The scattering problem for a sound-soft high aspect ratio ellipse is
\begin{equation}
    \begin{aligned}
         & \text{Find \( u =  u^\inc + u^\sca \in \Cscr^2(E) \cup
            \Cscr^1(\Rbb^2 \setminus {D}) \) such that: }
        \\
         & \begin{dcases}
            \Delta u + k^2 u = 0 & \text{in } E,          \\
            u = 0                & \text{on } \partial D, \\
            \lim \limits_{r \to \infty} \int_{|x| = r} \abs{\partial_n
                u^\sca - \ic k u^\sca}^2 \di{\sigma} = 0.
        \end{dcases}
    \end{aligned}\label{eq:sound-soft}
\end{equation}
Applying representation formula \cref{eq:u} to this problem, we obtain
\begin{equation}\label{eq:ud2}
    u(x) = u^\inc(x)
    + \int_{\partial D}G(x,y)  \partial_{n_y} u(y) \di{\sigma_{y}},
    \qquad x \in E,
\end{equation}
from which we determine that the unknown field \( \partial_{n_y} u \) on
the boundary satisfies
\begin{equation}\label{eq:ubie-soft}
    \frac{1}{2} \partial_{n_{x^b}}u(x^b) - \int_{\partial D} \partial_{n_{x^b}}
    G(x^b,y)  \partial_{n_{y}}u(y) \di{\sigma_{y}} = \partial_{n_{x^b}}u^\inc(x^b),
    \qquad x^b \in \partial D.
\end{equation}
When \( y = y(s) \) for \( s \in \Tbb \), we find that
\[
    \begin{aligned}
        \int_{\partial D}\partial_{n_{x^b}}G(x^b, y)\  \partial_{n_{y}}u(y)
        \di{\sigma_{y}}
         & = \int_{\Tbb} \partial_{n_x^b} G(y(s), y(t))\   \partial_{n_{y}}
        u(y(t))\ |y'(t)| \di{t},                                            \\
         & =  -  \frac{1}{ |y'(s)|}\int_{\Tbb}K(s, t; \eps )\  v(t) \di{t},
    \end{aligned}
\]
with \( v(t) \coloneqq \partial_{n_{y}} u(y(t))\ |y'(t)| \), and \( K \) given in
\cref{eq:KL}. Then \cref{eq:ubie-soft} becomes
\begin{equation}\label{eq:ubie-soft2}
    \frac{1}{2} v(s)  +\int_{\Tbb}K(s, t; \eps )\  v(t) \di{t} = v^\inc(s),
    \qquad s \in \Tbb,
\end{equation}
with \( v^\inc(s) = \partial_{n_{x^b}}u^\inc(y(s)) |y'(s)| \). Instead of
\( \partial_n u \), we have used the smoother unknown \( v \) because it is
better behaved at the degenerate points for this problem in the limit
as \( \eps \to 0 \). Rewriting \cref{eq:ubie-soft2} using our previous
operator notation, we have \( v \) satisfing \( \Lscr^H_D[v] = v^\inc \)
with \( \Lscr^H_D \coloneqq \frac{1}{2}\Irm + \Kscr^H \). The
resulting boundary integral equation is nearly identical to
\cref{eq:ubie}, except for the sign in front of the integral
operator.

Using \cref{lem:asy_KPsi} and~\cref{lem:asy_Kln} we write
\( \Lscr^H_D = \Hscr_0^D - \eps \, \Hscr_1 + \oo(\eps) \), with
\( \Hscr_0^D[v](s) = v_\odd(s) \), and \( \Hscr_1 \) defined in
\cref{eq:H_asy} from which we determine that the even part of \( v \) is
problematic. More precisely, ill-posedness is due to the even part
of \( v^{\inc} \) because it requires a different scaling.  Following
the same procedure as before, we compute the expansions of the
source term \( f \coloneqq v^\inc \) and find
\begin{subequations}
    \begin{align}
        f_0^\eve(s) & = 0, \qquad  f_1^\eve(s)  = \cos^2(s) \partial_{xx}
        u^\inc (0,\sin(s)) +  \sin(s) \partial_{y} u^\inc
        (0,\sin(s)), \label{eq:fd_ev}
        \\
        f_0^\odd(s) & = \cos(s) \partial_x u^\inc (0,\sin(s)) ,
        \qquad f_1^\odd(s)  = 0. \label{eq:fd_od}
    \end{align}
\end{subequations}
Using those expansions, we then compute
\begin{subequations}\label{eq:system_parity_solve_scatt_2_dirichlet}
    \begin{align}
        v_\eve^\asy(s)
         & = -  \Hscr_1^{-1} \left[ \cos^2(s) \partial_{xx} u^\inc
            (0,\sin(s)) +  \sin(s) \partial_{y} u^\inc (0,\sin(s)) \right],
        \\[
        1ex]
        v_\odd^\asy(s)
         & =  \cos(s) \partial_x u^\inc (0,\sin(s))  +  \eps \, \Hscr_1
        \left[ \cos(s) \partial_x u^\inc(0, \sin(s)) \right].
    \end{align}
\end{subequations}
Results from applying this procedure for this sound-soft problem are
shown in \cref{fig:dirichlet} and
\cref{fig:dirichlet_plane_wave}. They illustrate the different
asymptotic behaviors of even and odd source terms. Moreover, they
demonstrate the efficacy of QPAX to solve this problem.

\begin{figure}[hbt]
    \centering
    \begin{subfigure}{0.49\textwidth}
        \centering
        \includegraphics{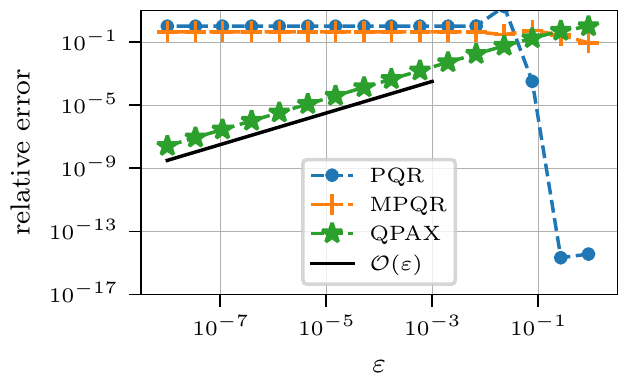}
        \caption{Even source: \( u^\inc(x, y) = \Mce{1}_3(\xi, q) \ce_3(\eta, q) \)}
    \end{subfigure}
    \begin{subfigure}{0.49\textwidth}
        \centering
        \includegraphics{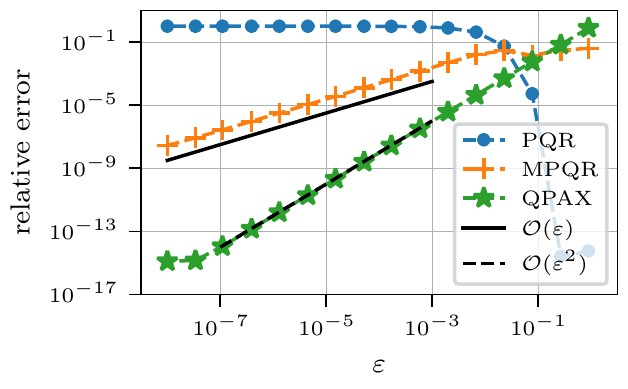}
        \caption{Odd source: \( u^\inc(x, y) = \Mse{1}_2(\xi, q) \se_2(\eta, q) \)}
    \end{subfigure}
    \caption{Results for relative error as a function of \( \eps \) in
        the numerical solution of the boundary integral equation for
        scattering by a sound-soft, high aspect ratio ellipse for PQR
        (blue curves with ``\( \bullet \)'' symbols), MPQR (orange
        curves with ``+'' symbols), and QPAX (green curves with
        ``\( \star \)'' symbols).  All of these results were computed
        using \( 2N = 64 \) quadrature points. The left plot shows
        results for an incident field that produces an even source on
        the boundary and the right plot shows results for an incident
        field that produces an odd source on the boundary.  The black
        line represents the \( \OO(\eps) \) convergence slope and the
        dashed black line represents the \( \OO(\eps^2) \) convergence
        slope.}%
    \label{fig:dirichlet}
\end{figure}

\begin{figure}[!hbt]
    \centering
    \begin{subfigure}{0.49\textwidth}
        \centering
        \includegraphics{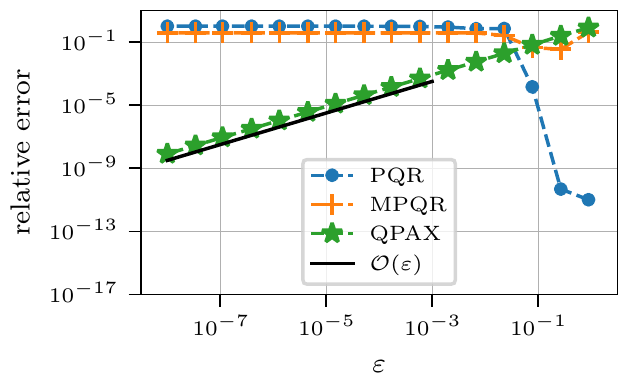}
        \caption{Approximate plane wave \cref{eq:u_inc_mat}}%
        \label{fig:dira}
    \end{subfigure}
    \begin{subfigure}{0.49\textwidth}
        \centering
        \includegraphics{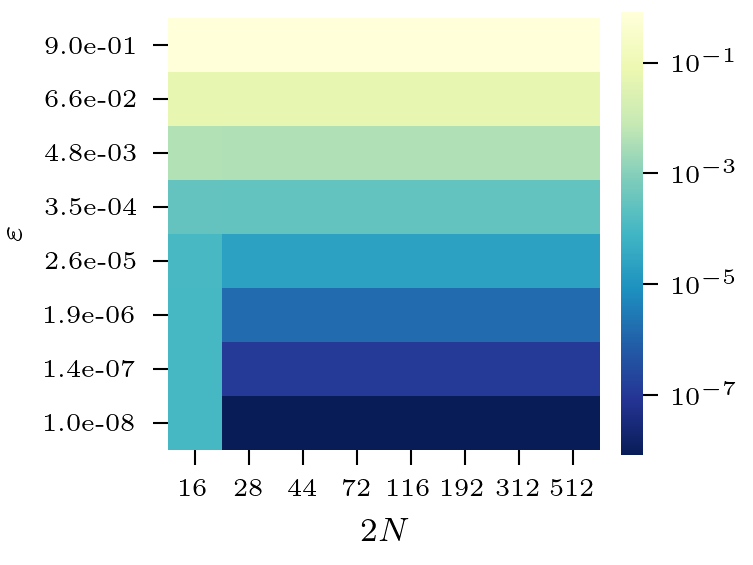}
        \caption{Relative error for QPAX}%
        \label{fig:dirb}
    \end{subfigure}
    \caption{(Left) Relative error as in \cref{fig:dirichlet}, but for the
        approximate plane wave incident field given in
        \cref{eq:u_inc_mat}.  (Right) The relative error for QPAX (log
        scale) plotted as a function of the number of quadrature points,
        \( 2N \), and the high aspect ratio parameter, \( \eps \), for
        the approximate plane wave incident field given in
        \cref{eq:u_inc_mat}.}%
    \label{fig:dirichlet_plane_wave}
\end{figure}

\subsection{Penetrable high aspect ratio ellipse}

Consider a penetrable high aspect ratio ellipse characterized by some (optical) property \( \varsigma_- \in \mathbb{C} \).
Typically, \( \varsigma_- \) represents the inverse of the permittivity for that particular medium.
It can be a plasmonic medium (noble metal such as gold or silver) with \( \Re\plr{\varsigma_-} < 0 \), or a dielectric medium with \( \Re\plr{\varsigma_-} > 0 \).
The ellipse is surrounded by vacuum (characterized by \( \varsigma_+ = 1 \)).
The scattering problem for this penetrable ellipse is given by the following transmission problem:
\begin{equation}\label{eq:penetrable}
    \begin{aligned}
         & \text{Find \( u_+  \in \Cscr^2(E) \cup \Cscr^1(\Rbb^2 \setminus
            {D}) \), \( u_- \in \Cscr^2(D) \cup \Cscr^1(\bar{D}) \) such that:
        }
        \\
         & \begin{dcases}
            \Delta u_+ + k^2_+ u_+ = 0                                   & \text{in } E,          \\
            \Delta u_- + k_-^2 u_- = 0                                   & \text{in } D,          \\
            u_+ = u_-                                                    & \text{on } \partial D, \\
            \varsigma_+ \partial_n u_+ =      \varsigma_- \partial_n u_- & \text{on } \partial D, \\
            \lim \limits_{r \to \infty} \int_{|x| = r} \abs{\partial_n
                u^\sca_+ - \ic k_+ u^\sca_+}^2 \di{\sigma} = 0.
        \end{dcases}
    \end{aligned}
\end{equation}
with \( k_\pm = k/\sqrt{\varsigma_\pm} \) and \( u_\pm = u^\inc + u^\sca_\pm \). We assume that \( (\varsigma_-, \varsigma_+) \) are such that the problem is well-posed.

Using the representation formula and the transmission conditions in \cref{eq:penetrable}, we find that, for \( x \in E \) the exterior total field is
\begin{subequations}\label{eq:u_penetrable}
    \begin{equation}
        u_+(x) = u^\inc(x)
        + \int_{\partial D} \partial_{n_y} G^+(x,y)  u_+(y) \di{\sigma_{y}} -   \int_{\partial D}G^+(x,y)  \partial_{n_y} u_+(y) \di{\sigma_{y}}
    \end{equation}
    and for \( x \in D \) the interior field is
    \begin{equation}
        u_-(x) =- \int_{\partial D} \partial_{n_y} G^-(x,y)  u_+(y) \di{\sigma_{y}} + \frac{\varsigma_+}{\varsigma_-}\int_{\partial D}G^-(x,y)  \partial_{n_y} u_+(y) \di{\sigma_{y}}
    \end{equation}
\end{subequations}
with \( G^\pm(x,y) = \frac{\ic}{4} \Ho_0(k_\pm |x-y|) \).
When \( y = y(s) \) for \( s \in \Tbb \), we find that the unknown fields on the boundary, \( (u_+(s), v_+(s)) \coloneqq (u_+(y(s)), \partial_{n_y} u_+(y(s))|y'(s)|) \), satisfy the system
\begin{equation}\label{eq:ubie-penetrable}
    \begin{pmatrix}
        \frac{1}{2}\Irm - \Kscr^H_+ & \Sscr^H_+
        \\[1ex]
        \frac{1}{2}\Irm + \Kscr^H_- & - \frac{\varsigma_+}{\varsigma_-} \Sscr^H_-
    \end{pmatrix}
    \begin{bmatrix}
        u_+(s) \\[1ex]
        v_+(s)
    \end{bmatrix}
    = \begin{bmatrix}
        u^\inc(s) \\[1ex]
        0
    \end{bmatrix}, \quad s \in \Tbb,
\end{equation}
with
\( \Sscr^H_\pm [v](s) = \displaystyle \frac{\ic}{4} \int_{\Tbb}
\Ho_0(k_\pm \ r(s, t; \eps ))\ v(t) \di{t} \), and \( \Kscr^H_\pm \)
denoting the operator \( \Kscr^H \) with either \( k = k_\pm \).

For a high aspect ratio ellipse, \cref{eq:ubie-penetrable} contains both double-layer potentials defined in \cref{eq:ubie,eq:ubie-soft2}.
Consequently, \cref{eq:ubie-penetrable} combines features from the sound-hard and the sound-soft cases indicating that there is (i) an underlying nearly singular behavior to address, and (ii) different parity scalings are necessary to accurately compute the field on the boundary.
System \cref{eq:ubie-penetrable} also has the single-layer potentials \( \Sscr^H_\pm \).
These single-layer potentials do not contribute to the nearly singular behavior in the limit \( \eps \to 0 \) at the mirror points \( s + t_s \equiv \pi \, [2\pi] \), however \( \Sscr^H_\pm [v](s) \) are weakly singular integrals on \( s = t \).
This singularity is well understood (see~\cite{Kress2014}), and one may use PQR to address that weak singularity.
To address the inherent parity issues associated with the double-layer potentials in \cref{eq:ubie-penetrable}, we may compute asymptotic expansions using \cref{lem:asy_KPsi} and \cref{lem:asy_Kln}.
With those expansions, we may properly scale \cref{eq:ubie-penetrable} and then apply a straight-forward extension of QPAX to compute an approximation solution.

\subsection{More general high aspect ratio particles}

Thus far, we have only considered a high apsect ratio ellipse because its simple explicit parameterization allows for a complete analysis of the problem.
We now consider a more general high aspect ratio particle whose closed, smooth boundary can be parameterized by a \( T \)-periodic curve: \( y(t) = (\eps y_1(t), y_2(t)) \), for \( t \in \Tbb_T \) with \( \Tbb_T \coloneqq \Rbb / T\Zbb \).
While computations for general cases are necessarily more cumbersome, we show that the parity issues we have identified for an ellipse generalize in an intuitive way.

Consider scattering by a sound-hard high aspect ratio particle
\( D \) (not necessarily symmetric) shown in \cref{fig:shape_hyp}.
Boundary integral equation \cref{eq:ubie} is to be solved.  Using
the parameterization of the boundary given above, \cref{eq:ubie} is
\( \frac{1}{2} u(s) - \Kscr^H[u](s) = u^\inc(s) \), for \( s \in
\Tbb_T \).
We find that the kernel in \( \Kscr^H \) admits the following factorization,
\begin{equation}\label{eq:kernel-general}
    \begin{aligned}
         & \Kscr^H[u](s) = \frac{\ic k}{4} \int_\Tbb \abs{y'(t)} \frac{n(y(t)) \cdot (y(s)-y(t))}{\abs{y(s)-y(t)}} \Ho_1\plr{k \, \abs{y(s)-y(t)}} \, u(t) \di{t},
        \\
         & = \frac{\ic k \pi}{2} \int_\Tbb  \underset{ = K^L(s,t;\eps)}{\underbrace{\left[\frac{1}{2\pi}\abs{y'(t)} \frac{n(y(t)) \cdot (y(s)-y(t))}{\abs{y(s)-y(t)}^2} \right]}} \Ho_1\plr{k \, \abs{y(s)-y(t)}} \abs{y(s)-y(t)}\, u(t) \di{t}.
    \end{aligned}
\end{equation}
Written more explicitly, we have
\[
    \begin{aligned}
         & K^L(s,t;\eps)  =\frac{\eps}{2\pi}\frac{y'_2(s)(y_1(s)-y_1(t))-y'_1(s)(y_2(s)-y_2(t))}{\eps^2 {(y_1(s)-y_1(t))}^2 + {(y_2(s)-y_2(t))}^2} ,\quad s,t \in \Tbb_T.
    \end{aligned}
\]
Just as with the ellipse, \( K^{L} \) is the kernel in the double-layer
potential for Laplace's equation
for this boundary.
From previous results, we know that nearly singular behaviors arise at mirror points and weakly singular behaviors at degenerate points.
The nearly singular behaviors require parity treatment to accurately compute the field.

We require a generalization of mirror points and degenerate
points. To that end, we introduce the function
\( \sigma : \Tbb_T \to \Tbb_T \). A mirror point is one that
satisfies \( y_2(\sigma(t)) = y_2(t) \), \( t \in \Tbb_T \).
Let \( y(0) \) and \( y(t_\star) \), for \( t_\star \in (0, T) \), denote the ``bottom'' and ``top'' points of \( \partial D \), respectively, where \( y_2(0) = \min_{\Tbb_T} y_2 \), \( y_2(t_\star) = \max_{\Tbb_T} y_2 \), and \( y_2'(0) = 0 = y_2'(t_\star) \).
It follows that the set \( \{0, t_\star \} \) represents the degenerate points in the sense that \( \sigma (0) \equiv 0 \, [T] \), \( \sigma (t_\star ) \equiv t_\star \, [T] \) (see \cref{fig:shape_hyp}).

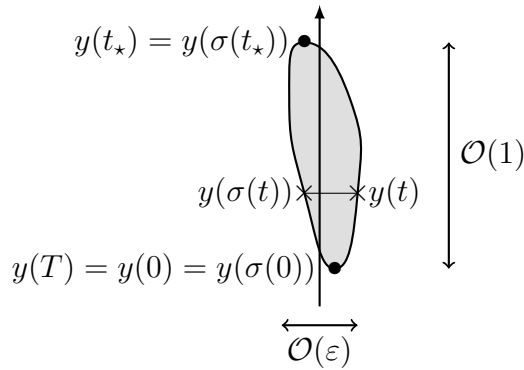
\begin{figure}[!hbtp]
    \centering
    \begin{tikzpicture}
        \filldraw [draw=black,thick,fill=gray!25]
        plot [smooth cycle, tension=1] coordinates {
                (-0.2,-0.5)
                (0.2,-1.5)
                (0.5,-0.5)
                (0.4,0.75)
                (-0.2,1.5)
                (-0.4,0.75)
            };
        \draw[thick,->,>=latex] (0,-2) -- (0,2);
        \draw[thick,<->] (-0.5,-2.25) -- (0.5,-2.25);
        \draw (0,-2.25) node [below] {\( \OO(\eps) \)};
        \draw[thick,<->] (1.7,-1.5) -- (1.7,1.5);
        \draw (1.7,0) node [right] {\( \OO(1) \)};
        \draw (0.2,-1.5) node {\( \bullet \)};
        \draw (0.1,-1.5) node [left] {\( y(T) = y(0) = y (\sigma(0)) \)};
        \draw (-0.2,1.5) node {\( \bullet \)};
        \draw (-0.25,1.5) node [left] {\( y(t_\star) = y(\sigma(t_\star)) \)};
        \draw (-0.2,-0.5) node {\( \times \)} -- (0.5,-0.5) node {\( \times \)};
        \draw (-0.2,-0.5) node [left] {\( y(\sigma(t)) \)};
        \draw (0.5,-0.5) node [right] {\( y(t) \)};
    \end{tikzpicture}
    \caption{Scheme representing a general high aspect ratio particle with considered shape assumptions.}%
    \label{fig:shape_hyp}
\end{figure}

Just like we found for the ellipse, nearly singular behavior
arises at mirror points:
\begin{align*}
    K^L(s,t;\eps)
     & \sim \frac{1}{ 2 \pi \eps} \frac{y_2'(\sigma(s))}{y_1(s)-y_1(\sigma(s))},
     &
     & \text{as } t \to \sigma(s) \text{ for } s \in \Tbb_T \setminus
    \{ 0, t_\star \}.
\end{align*}
Additionally, weakly singular behavior arises at the degenerate points:
\begin{align*}
    K^L(s,t;\eps)
     & \sim -\frac{1}{4 \pi \eps} \frac{y_2''(\sigma(s))}{y_1'(\sigma(s))},
     &
     & \text{as } t \to \sigma(s) \text{ for } s \in \{ 0, t_\star \},
\end{align*}
and the derivative of the kernel in \cref{eq:kernel-general} admits a singularity for \(s \in \{ 0, t_\star \} \).
The more general high aspect ratio particle has \emph{exactly the same structural features} as the ellipse. More precisely, at first order we have \( \frac{1}{2} u(s) - \Kscr^H[u](s) = u_\eve(s) + \OO(\eps) \). In that regard, we make the following two remarks.
\begin{itemize}
    \item As long as one can define mirror points with respect to the
          major axis of the obstacle \( D \), one can always factor out Laplace's
          kernel \( K^{L} \) and identify its nearly singular behavior at
          those mirror points.
    \item To address the nearly singular behavior, one can proceed as
          presented in \cref{sec:QPAX}. The decomposition
          \( \Kscr^H = \Kscr^\Psi + \Kscr^{\ln} \) with
          \( \Kscr^\Psi,\Kscr^{\ln} \) defined in \cref{eq:KPsi},
          \cref{eq:Kln} remains valid, and similar results as in
          \cref{lem:asy_KPsi}, \cref{lem:asy_Kln} can be derived. That means that
          even and odd parts still require a different scaling, and one can
          use QPAX to compute accurately the solution.
\end{itemize}

\section{Conclusions}%
\label{sec:conclusion}

We have studied two-dimensional sound-hard scattering by a high aspect ratio ellipse using boundary integral equations.
In the limit as \( \eps \to 0^+ \) (where \( \eps^{-1} \) denotes the aspect ratio), the integral operator in the boundary integral equation exhibits nearly singular behavior corresponding to the collapsing of the ellipse to a line segment.
This nearly singular behavior leads to large errors when solving the boundary integral equation.

By performing an asymptotic analysis of the integral operator for a high aspect ratio ellipse we find that two different scalings due to parity are needed to solve the boundary integral equation.
By introducing two different asymptotic expansions for even and odd parity solutions, we identify the leading behavior of the solution.
We first study the Dirichlet problem for Laplace's equation in a high aspect ratio ellipse to identify these behaviors.
Upon determining the correct asymptotic expansions for the different parities, we introduce a numerical method which we call Quadrature by Parity Asymptotic eXpansions (QPAX) that effectively and efficiently solves this problem. We then extend QPAX for the scattering problem.

Our results show that the relative error (defined with respect to the
infinity norm) exhibits an \( \OO(\eps^{2}) \) behavior for even parity
problems, and an \( \OO(\eps) \) behavior for odd parity problems.
Thus, in general, the relative errror for QPAX exhibits an
\( \OO(\eps) \) behavior.  In contrast, other methods we have used that
do not explicitly take into account the parity of solutions may or may
not work for even parity problems, but do not work for odd parity
problems.

By identifying the different asymptotic behaviors due
to parity and developing an effective and efficient method to
compute them, we have identified and addressed an elementary issue
in two-dimensional scattering by high aspect ratio particles.
Scattering by a sound-hard, high aspect ratio ellipse contains all
important features needed to tackle other problems. We have shown
that the high aspect ratio sound-soft and penetrable ellipse
have the same asymptotic behaviors. Moreover, we have shown that
scattering by a more general high aspect ratio particle also yields
two distinct asymptotic behaviors that must be addressed carefully.

It is likely that three-dimensional scattering problems for high
aspect ratio particles share the features that we have identified
here for two-dimensional scattering problems. It is certainly true
that axisymmetric high aspect ratio particles will yield the two
different asymptotic behaviors seen here since they effectively
reduce down to a two-dimensional problem. Interesting results have been obtained in the context of characterizing plasmonic resonances in slender-bodies~\cite{ruiz2019slender,deng2021mathematical,ando2020spectral}.
The exact details for a
general three-dimensional high aspect ratio particle remain to be
determined and constitute future work. Nonetheless, the results obtained here provide valuable
insight into that problem. For this reason, we believe that these
results are useful for further studies of scattering by high aspect
ratio particles and related applications.


\appendix

\section{Analytical solution for scattering by a sound-hard ellipse}%
\label{sec:Appendix-A}

Boundary value problem \cref{eq:sound-hard} can be solved analytically as follows.
We define the elliptical coordinates
\( (\xi, \eta) \in \Rbb_+^* \times \Tbb \) as
\[
    x = c_\eps \, \sinh(\xi)\, \sin(\eta) \quad
    \text{and} \quad
    y = c_\eps \, \cosh(\xi)\, \cos(\eta),
\]
where \( c_\eps = \sqrt{1-\eps^2} \) is a parameter defining the focus of the ellipse.
The boundary in these coordinates is \( \{ \xi_\eps \} \times \Tbb \), where \( \xi_\eps = \argth(\eps) \).
Using these elliptical coordinates, the Helmholtz equation becomes
\[
    (\partial_{\xi\xi} + \partial_{\eta\eta}) u^\sca + \frac{c_\eps^2
        k^2}{2} (\cosh(2\xi) - \cos(2\eta)) u^\sca = 0,
\]
it admits two sets of solutions: \( {\{ \Mce{i}_m(\xi, q)\ce_m(\eta, q) \}}_{m \ge 0} \), \( {\{ \Mse{i}_m(\xi, q)\se_m(\eta, q) \}}_{m \ge 1} \) with \( q = c_{\eps}^{2} k^{2}/4 \) and \( m \) integer.
Here, \( \ce \) and \( \se \) denote the angular Mathieu functions of order \( m \), and \( \Mce{i} \) and \( \Mse{i} \) denote the radial Mathieu function of order \( m \) and \( i^{\text{th}} \)-kind. Using~\cite[Sec.~28.20(iii)]{Nist}, we find that only \( \Mce{3}_m \) and \( \Mse{3}_m \) satisfy the Sommerfeld radiation condition.
Thus, the scattered field is given by
\begin{equation}\label{eq_u_ana}
    u^\sca(\xi, \eta) = \sum_{m = 0}^{\infty} \alpha_m\Mce{3}_m(\xi, q)
    \ce_m(\eta, q) + \sum_{m = 1}^{\infty} \beta_m\Mse{3}_m(\xi, q)
    \se_m(\eta, q),
\end{equation}
with the coefficients \( \alpha_{m} \) and \( \beta_{m} \) to determine by requiring a sound-hard boundary condition on \( \partial D \) to be satisfied.
For a given incident field, \( u^\inc(\xi_{\eps},\eta) \), we have for the sound-hard problem,
\[
    \alpha_m = \frac{-1}{{\Mce{3}_m}'(\xi_\eps, q)\, \pi} \int_0^{2\pi}
    \sqrt{{\sin(\eta)}^2 + \eps^2 {\cos(\eta)}^2} \,
    \partial_{n}u^\inc(\xi_{\eps},\eta) \,
    \ce_m(\eta, q) \di{\eta}
\]
and
\[
    \beta_m = \frac{-1}{{\Mse{3}_m}'(\xi_\eps, q)\, \pi} \int_0^{2\pi}
    \sqrt{{\sin(\eta)}^2 + \eps^2 {\cos(\eta)}^2} \,
    \partial_{n}u^\inc(\xi_{\eps},\eta) \,
    \se_m(\eta, q) \di{\eta}.
\]

\section{Matched asymptotic expansions}%
\label{sec:Appendix-B}

In this Appendix we provide matched asymptotic expansions of the double-layer potential for Laplace and Helmholtz equation.

\subsection{Proof of \texorpdfstring{\cref{lem:asy_KL}}{\ref{lem:asy_KL}}}%
\label{ssec:KL_asy}

\begin{proof}
    We rewrite \cref{eq:int_op_Kl} as
    \[
        \Kscr^L[\mu](s) = -\frac{\eps}{2\pi} \int_{-s}^{2 \pi -s}
        \frac{\mu(t)}{1+\eps^2 - (1-\eps^2)\cos(s +t)} \di{t},
    \]
    so that when using the change of variable \( t = \pi-s+x \) we have
    \[
        \Kscr^L[\mu](s) = -\frac{\eps}{2\pi} \int_{-\pi}^\pi
        \frac{\mu(\pi-s+x)}{1+\eps^2 - (1-\eps^2)\cos(x)} \di{x}.
    \]
    The above operator is then a nearly singular integral about \( x \equiv 0 \, [2\pi] \).
    Introducing \( \delta > 0 \) such that \( \eps = \oo(\delta) \), we split the integral \( \Kscr^L[\mu](s) = \Kscr^{\mathrm{inn}}[\mu](s) + \Kscr^{\mathrm{out}}[\mu](s) \) with
    \begin{subequations}\label{eq:io_term}
        \begin{align}
            \Kscr^{\mathrm{inn}}[\mu](s)
             & \coloneqq \frac{-\eps}{2\pi}
            \int_{-\frac{\delta}{2}}^\frac{\delta}{2}
            \frac{\mu(\pi-s+x)}{1+\eps^2 - (1-\eps^2)\cos(x)} \di{x},
            \\
            \Kscr^{\mathrm{out}}[\mu](s)
             & \coloneqq \frac{-\eps}{2\pi} \int_{I_\delta}
            \frac{\mu(\pi-s+x)}{1+\eps^2 - (1-\eps^2)\cos(x)} \di{x},
        \end{align}
    \end{subequations}
    with \( I_\delta = (-\pi, \pi) \setminus \plr{-\tfrac{\delta}{2}, \tfrac{\delta}{2}} \) and we look for the leading order terms of the above integrals as \( \delta, \eps \to 0^+ \).
    Note that the inner term \( \Kscr^{\mathrm{inn}}[\mu](s) = I_\delta(\eps) \) defined in \cref{sec:nearly-singular-laplace}, then following the same procedure (using the change of variable \( x = \eps X \) and expanding about \( \eps = 0 \) then \( \delta = 0 \)), we obtain
    \begin{align}
        \Kscr^{\mathrm{inn}}[\mu](s)
         & = -\frac{1}{\pi} \arctan\left(\frac{\delta}{4\eps}\right)
        \mu(\pi-s) + \OO(\delta \eps) \nonumber
        \\
         & = -\frac{1}{2} \mu(\pi-s) + \frac{4}{\pi} \mu(\pi-s)
        \frac{\eps}{\delta} + \OO\left(\frac{\eps^3}{\delta^3}\right) +
        \OO(\delta \eps). \label{eq:inn_expan}
    \end{align}
    For the outer term \( \Kscr^{\mathrm{out}}[\mu](s) \), we first expand the integrand for \( x \not \equiv 0 \, [2\pi] \), leading to
    \begin{align}\label{eq:out_e_expan}
        \Kscr^{\mathrm{out}}[\mu](s)
         & = \frac{-\eps}{2\pi} \int_{I_\delta}
        \frac{\mu(\pi-s+x)}{1 - \cos(x)} \left[ 1 +
            \OO\left(\frac{\eps^2}{{\tan(\tfrac{x}{2})}^2}\right) \right]
        \di{x} \nonumber
        \\
         & = \frac{-\eps}{2\pi} \int_{I_\delta} \frac{\mu(\pi-s+x)}{1
            - \cos(x)} \di{x} + \OO\left(\frac{\eps^3}{\delta^3}\right).
    \end{align}
    Let us note that \( x \mapsto {(1 - \cos(x))}^{-1} \) is an even function (in the classic sense) over \( I_\delta \). We define \( \nu(x) = \mu(\pi-s+x) \) and we split \( \nu = \nu_+ + \nu_- \) into its \emph{classic} even and odd parts
    \( \nu_\pm = \frac{1}{2}(\nu(x) \pm \nu(-x)) \), we then obtain
    \begin{align*}
        \int_{I_\delta} \frac{\nu(x)}{1 - \cos(x)} \di{x}
         & = \int_{I_\delta} \frac{\nu_+(x)}{1 - \cos(x)} \di{x}
        \\
         & = \nu_+(0) \int_{I_\delta} \frac{1}{1 - \cos(x)} \di{x}
        + \int_{I_\delta} \frac{\nu_+(x) - \nu_+(0)}{1 - \cos(x)} \di{x}
        \\
         & = \frac{2 \nu_+(0)}{\tan(\tfrac{\delta}{4})}
        + \int_{I_\delta} \frac{\nu_+(x) - \nu_+(0)}{1 - \cos(x)} \di{x}
    \end{align*}
    Note that \( \mathsf{E}_{\pi-s}[\mu] \) defined in \cref{eq:Es} extends the function
    \[
        x \mapsto \frac{\nu_+(x) - \nu_+(0)}{1 - \cos(x)}
        = \frac{\mu(\pi -s + x) - 2 \mu (\pi -s)+\mu(\pi -s -x)}{2(1 - \cos(x))}
    \]
    to a continuous function in \( \Cscr(\Tbb) \).
    We write (using some rescaling and expansion about \( \eps = 0 \))
    \begin{align*}
        \int_{I_\delta} \frac{\nu_+(x) - \nu_+(0)}{1 - \cos(x)} \di{x}
         & =  \int_{-\pi}^\pi  \mathsf{E}_{\pi -s}[\mu](x) \di{x}
        - \int_{-\frac{\delta}{2}}^\frac{\delta}{2} \mathsf{E}_{\pi-s}[\mu](x) \di{x}
        \\
         & =  \int_{-\pi}^\pi  \mathsf{E}_{\pi -s}[\mu](x) \di{x} +
        \OO(\delta)
    \end{align*}
    because \( \mathsf{E}_{\pi -s}[\mu] \) is bounded on
    \( \plr{-\frac{\delta}{2}, \frac{\delta}{2}} \).  Using the fact
    that
    \[
        \frac{2}{\tan(\tfrac{\delta}{4})}\nu_+(0) =
        8\mu(\pi-s)\, \frac{1}{\delta} + O(\delta),
    \] combining all of the
    above results we obtain
    \begin{equation}\label{eq:out_expan}
        \Kscr^{\mathrm{out}}[\mu](s)
        = -\frac{4}{\pi}\mu(\pi-s) \, \frac{\eps}{\delta}
        -\frac{\eps}{2\pi} \int_{-\pi}^\pi \mathsf{E}_{\pi -s}[\mu](x)\di{x}
        + \OO\left(\frac{\eps^3}{\delta^3}\right) + \OO(\delta \eps).
    \end{equation}
    Finally plugging \cref{eq:inn_expan} and \cref{eq:out_expan} into \( \Kscr^L[\mu](s) \), and choosing \( \delta \) such that \( \OO\left(\frac{\eps^3}{\delta^3}\right) + \OO(\delta \eps) = \oo(\eps) \) (for example \( \delta = \sqrt{\eps} \)) we obtain
    \[
        \Kscr^L[\mu](s) = -\frac{1}{2} \mu(\pi-s) -\frac{\eps}{2\pi}
        \int_{-\pi}^\pi \mathsf{E}_{\pi -s}[\mu](x) \di{x} + \oo(\eps).
    \]
\end{proof}

\subsection{Proof of \texorpdfstring{\cref{lem:asy_KPsi}}{\ref{lem:asy_KPsi}}}%
\label{ssec:Kpsi_asy}

\begin{proof}
    We have the expansion
    \begin{equation}\label{eq:ze_expan}
        z_\eps(s, t) = \begin{dcases}
            2\, k\, \left\lvert \sin\left(\tfrac{s-t}{2}\right)
            \cos\left(\tfrac{s+t}{2}\right) \right\rvert + \OO(\eps^2)
             & \text{if } s+t \not \equiv \pi \, [2\pi] \\
            2\, k\, |\cos(s)|\, \eps
             & \text{if } s+t \equiv \pi \, [2\pi]
        \end{dcases}
    \end{equation}
    using \( \Psi'(0) = 0 \), we get \( \Psi(z_\eps(s, t)) = \Psi\plr{ 2\, k\, \abs{\sin\plr{\tfrac{s-t}{2}} \cos\plr{\tfrac{s+t}{2}}} } + \OO(\eps^2) \).
    We obtain \( \Kscr^\Psi[\mu](s) = \Kscr^L[ \Psi\plr{ 2\, k\, \abs{\sin\plr{\tfrac{s-t}{2}} \cos\plr{\tfrac{s+t}{2}}} } \mu ] (s) + \OO\plr{ \eps^2 \Kscr^L[\mu](s) } \), then we apply \cref{lem:asy_KL} to \( \Kscr^L \) and use \( \Psi(0) = 1 \) to finish the proof.
\end{proof}

\begin{lemma}\label{lem:L1_spec}
    Consider \( \Lscr_1 \) defined in \cref{eq:opL_01}. For all \( m \in \Zbb \), we have
    \[
        \Lscr_1[\mathsf{e}^{\ic m t}](s) = -|m|\, \mathsf{e}^{\ic m (\pi-s)} .
    \]
\end{lemma}
\begin{proof}
    From the expression of \cref{eq:opL_01}, we obtain
    \[
        \Lscr_1[\mathsf{e}^{\ic m t}](s)
        = -\frac{\ex^{\ic m (\pi-s)}}{2\pi} \int_\Tbb {\left[ \frac{\sin\plr{\frac{|m| t}{2}}}{\sin\plr{\frac{t}{2}}} \right]}^2 \di{t}.
    \]
    We recognize the Fej\'er kernel and we get the result using~\cite[Eq.~(1.15.16)]{Nist}.
\end{proof}

\subsection{Proof of \texorpdfstring{\cref{lem:asy_Kln}}{\ref{lem:asy_Kln}}}%
\label{ssec:say_Kln}

\begin{proof}
    Recall from \cref{eq:Kln} we have
    \[
        \Kscr^{\ln}[\mu](s) = -\frac{1}{2} \int_\Tbb K^L(s, t; \eps) \
        z_\eps(s, t) \bJ_1(z_\eps(s, t)) \ln\plr{ \frac{4 {z_\eps(s, t)}^2}{k^2} } \ \mu(t) \di{t} .
    \]
    We start by showing that \( K^L(s, t; \eps) \, z_\eps(s, t) \bJ_1(z_\eps(s, t)) \) is regular as \( \eps \to 0^+ \).
    From the definition of \( \bJ_1 \), see~\cite[Section.~10.2(ii)]{Nist}, there exists an analytic function \( \Phi \) such that \( \bJ_1(z) = z\Phi(z) \) and \( \Phi(0) = \tfrac{1}{2} \).
    Using \( \Phi \) \cref{eq:zeps,eq:ze_expan}, we get after using Taylor expansions about \( \eps = 0 \)
    \begin{align*}
        K^L(s, t; \eps) \, z_\eps(s, t) \bJ_1(z_\eps(s, t))
         & = K^L(s, t; \eps) \, {z_\eps(s, t)}^2 \, \Phi(z_\eps(s, t))
        \\
         & = -\frac{\eps k^2}{\pi} \sin\plr{\tfrac{s-t}{2}}^2 \Phi\left(
        2\, k\, \abs{ \sin\left(\tfrac{s-t}{2}\right)
            \cos\plr{\tfrac{s+t}{2}} } \right) + \OO(\eps^2).
    \end{align*}
    We then rewrite
    \begin{equation}\label{eq:zeps_log}
        \ln\plr{ \frac{4 {z_\eps(s, t)}^2}{k^2} }
        = \ln\plr{ 4 \sin\plr{\tfrac{s-t}{2}}^2 }
        + \ln\plr{ 4 \cos\plr{\tfrac{s+t}{2}}^2 + 4 \eps^2
            \sin\plr{\tfrac{s+t}{2}}^2 }.
    \end{equation}
    Using \cref{lem:asy_ln} and the function \( \phi_s(t) = 2 k^2 \sin\plr{\tfrac{s-t}{2}}^2 \Phi\plr{ 2\, k\, \abs{ \sin\plr{\tfrac{s+t}{2}} \cos\plr{\tfrac{s+t}{2}} } } \) introduced in \cref{eq:Phi_expression} we obtain
    \begin{align*}
        \Kscr^{\ln}[\mu](s)
         & = \frac{\eps}{2\pi} \left[
            \int_\Tbb \phi_s(t) \mu(t) \ln\plr{ 4\sin\plr{\tfrac{s-t}{2}}^2 } \di{t}
            + \int_\Tbb \phi_s(t) \mu(t) \ln\plr{ 4 \cos\plr{\tfrac{s+t}{2}}^2 } \di{t}
            \right]
        \\
         & = \frac{\eps}{2\pi} \int_\Tbb {[\phi_s \, \mu]}_\eve(t) \,
        \ln\plr{ 4 \sin\plr{\tfrac{s-t}{2}}^2 } \di{t}
    \end{align*}
    where we used the change of variable \( t = \pi-x \) in the second integral, leading to the even representation of \( \phi_s \, \mu \).
\end{proof}

\begin{lemma}\label{lem:asy_ln}
    For \( f \in \Cscr^\infty(\Tbb) \), and \( s \in \Tbb \), we have the following asymptotic
    expansion
    \[
        \int_\Tbb f(t)\, \ln\plr{ 4 \cos\plr{\tfrac{s+t}{2}}^2 + 4 \eps^2
            \sin\plr{\tfrac{s+t}{2}}^2 } \di{t} = \int_\Tbb f(t)\, \ln\plr{
            4 \cos\plr{\tfrac{s+t}{2}}^2 } \di{t} + \oo(1)
    \]
    as \( \eps \to 0^+ \).
\end{lemma}

\begin{proof}
    Using the change of variable \( t = \pi-s+x \), we then split the integral
    \[
        \int_{-\pi}^\pi f(\pi-s+x)\, \ln\plr{ 4 \sin\plr{\tfrac{x}{2}}^2
            + 4 \eps^2 \cos\plr{\tfrac{x}{2}}^2 } \di{x} =
        \mathcal{F}^\mathrm{inn} + \mathcal{F}^\mathrm{out}
    \]
    with
    \begin{align*}
        \mathcal{F}^\mathrm{inn}
         & = \int_{-\delta}^\delta f(\pi-s+x)\,
        \ln\plr{ 4 \sin\plr{\tfrac{x}{2}}^2
        + 4 \eps^2  \cos\plr{\tfrac{x}{2}}^2 }\di{x}, \\
        \mathcal{F}^\mathrm{out}
         & = \int_{I_\delta} f(\pi-s+x)\,
        \ln\plr{ 4 \sin\plr{\tfrac{x}{2}}^2
            + 4 \eps^2 \cos\plr{\tfrac{x}{2}}^2 } \di{x}
    \end{align*}
    where \( I_\delta = [-\pi, \pi] \setminus (-\delta, \delta) \), and \( \delta \) is a parameter such that \( \delta \to 0^+ \) with \( \eps = \oo(\delta) \).
    We now use classic techniques from matched asymptotic expansions (see for instance \cref{ssec:KL_asy}).
    For the inner expansion, we remark that, as \( \eps \to 0^+ \), we have
    \[
        \abs{ \ln\plr{ 4 \sin\plr{\tfrac{x}{2}}^2
                + 4 \eps^2  \cos\plr{\tfrac{x}{2}}^2 } } \le 2\abs{\ln(2\eps)}
    \]
    which gives \( \mathcal{F}^\mathrm{inn} = \OO(\delta \abs{\ln(\eps)}) \).
    Then for the outer expansion, setting \( x = s + t - \pi \) and using the fact that
    \[
        \ln\plr{ 4 \sin\plr{\tfrac{x}{2}}^2
            + 4 \eps^2 \cos\plr{\tfrac{x}{2}}^2 } = \ln\plr{ 4
            \sin\plr{\tfrac{x}{2}}^2 } + \OO\plr{ \eps^2
            \tan\plr{\tfrac{x}{2}}^{-2} }
    \]
    we obtain
    \begin{align*}
        \mathcal{F}^\mathrm{out}
         & = \int_{I_\delta} f(\pi-s+x)\, \ln\plr{ 4
            \sin\plr{\tfrac{x}{2}}^2 } \di{x} +
        \OO\plr{\frac{\eps^2}{\delta^3}}
        \\
         & = \int_{-\pi}^\pi f(\pi-s+x)\, \ln\plr{ 4
            \sin\plr{\tfrac{x}{2}}^2 } \di{x}
        + \OO\plr{\delta \abs{\ln{\delta}}} +
        \OO\plr{\frac{\eps^2}{\delta^3}}
    \end{align*}
    One concludes by remarking that we can choose \( \delta \to 0 \) (for example \( \delta = \sqrt{\eps} \)) such that \( \OO\plr{\delta \abs{\ln(\eps)}} + \OO\plr{\delta \abs{\ln{\delta}}} + \OO\plr{\frac{\eps^2}{\delta^3}} = \oo(1) \).
\end{proof}

\bibliographystyle{siamplain}
\bibliography{references}

\end{document}